\documentclass{article}

\newtheorem{lemma}[subsection]{Lemma}
\newtheorem{proposition}[subsection]{Proposition}
\newtheorem{theorem}[subsection]{Theorem}
\newtheorem{corollary}[subsection]{Corollary}
\newtheorem{fact}[subsection]{Fact}
\newenvironment{question}%
   {\refstepcounter{subsection}%
        \medbreak\noindent{\bf Question \thesubsection\space}}%
   {\par\medbreak}%
\newenvironment{remark}%
   {\refstepcounter{subsection}%
        \medbreak\noindent{\bf Remark \thesubsection\space}}%
   {\par\medbreak}%
\newenvironment{example}%
   {\refstepcounter{subsection}%
        \medbreak\noindent{\bf Example \thesubsection\space}}%
   {\par\medbreak}%
\newenvironment{proof}%
   {\medbreak\noindent{\it Proof:\space}}%
   {\par\noindent\vrule height 5pt width 5pt depth 0pt\smallbreak}%
\newcommand{\df}{\bf}
\let\sauvegardetiret=\-
\renewcommand{\-}[1]{\ifx#1-\penalty10000\hbox{-\relax}\penalty10000\else\sauvegardetiret#1\fi}
\long\def\footnotesymbol[#1]#2{\begingroup%
\def\thefootnote{\fnsymbol{footnote}}\footnote[#1]{#2}\endgroup}
\long\def\footnotestar#1{\begingroup%
\def\thefootnote{\fnsymbol{footnote}}\footnote[0]{#1}\endgroup}
\newcommand{\tq}{\,\big/\ }
\newcommand{\llor}{\mathop{\lor\mskip-6mu\relax\lor}}
\newcommand{\vvee}{\llor}
\newcommand{\lland}{\mathop{\land\mskip-6mu\relax\land}}
\newcommand{\wwedge}{\lland}
\newcommand{\mmeet}{\wwedge}
\newcommand{\jjoin}{\vvee}
\newcommand{\meet}{\wedge}
\newcommand{\join}{\vee}
\newcommand{\NN}{{\rm\bf N}}
\newcommand{\ZZ}{{\rm\bf Z}}

\newcommand{\RR}{{\rm\bf R}}

\usepackage{amssymb}
\usepackage[arrow,curve,matrix,tips,frame]{xy}
\usepackage[dvips,bookmarksnumbered=true]{hyperref}
\let\sauvegardeuparrow=\uparrow
\renewcommand{\uparrow}{\mathord\sauvegardeuparrow}
\let\sauvegardedownarrow=\downarrow
\renewcommand{\downarrow}{\mathord\sauvegardedownarrow}
\let\sauvegardeUparrow=\Uparrow
\renewcommand{\Uparrow}{\mathord\sauvegardeUparrow}
\let\sauvegardeDownarrow=\Downarrow
\renewcommand{\Downarrow}{\mathord\sauvegardeDownarrow}
\newcommand{\conj}{\bigwedge}
\newcommand{\cconj}{\mathop{\conj\mskip-9mu\relax\conj}}
\newcommand{\disj}{\bigvee}
\newcommand{\ddisj}{\mathop{\bigvee\mskip-9mu\relax\bigvee}}

\newcommand{\llat}{{\cal L}_{\rm lat}}
\newcommand{\ltc}{{\cal L}_{\rm HA^*}}
\newcommand{\lha}{{\cal L}_{\rm HA}}
\newcommand{\ii}{{\mathfrak i}}
\newcommand{\p}{{\mathfrak p}}
\newcommand{\q}{{\mathfrak q}}
\newcommand{\UN}{{\rm\bf 1}}
\newcommand{\ZERO}{{\rm\bf 0}}
\newcommand{\cF}{{\cal F}}
\newcommand{\cO}{{\cal O}}
\newcommand{\cP}{{\cal P}}
\newcommand{\cS}{{\cal S}}
\newcommand{\cV}{{\cal V}}
\newcommand{\Var}{{\rm Var}}
\newcommand{\Spec}{\mathop{\rm Spec}\nolimits}
\newcommand{\jirr}{\mathop{{\cal I}^\join}\nolimits}
\newcommand{\mirr}{\mathop{{\cal I}^\meet}\nolimits}
\newcommand{\cjirr}{\mathop{{\cal I}^{!\join}}\nolimits}
\newcommand{\cmirr}{\mathop{{\cal I}^{!\meet}}\nolimits}
\newcommand{\jsupp}[1]{\mathop{{\rm Comp}^\join_{#1}}\nolimits}
\newcommand{\msupp}[1]{\mathop{{\rm Comp}^\meet_{#1}}\nolimits}

\newcommand{\Th}{\mathop{\rm Th}\nolimits}
\newcommand{\hgt}{\mathop{\rm height}\nolimits}
\newcommand{\cohgt}{\mathop{\rm coheight}\nolimits}
\newcommand{\codim}{\mathop{\rm codim}\nolimits}
\newcommand{\cdim}{\mathop{\rm (co)dim}\nolimits}
\newcommand{\rk}{\mathop{\rm rk}\nolimits}
\newcommand{\cork}{\mathop{\rm cork}\nolimits}
\newcommand{\crk}{\mathop{\rm (co)rk}\nolimits}
\newcommand{\Ker}{\mathop{\rm Ker}\nolimits}
\newcommand{\dist}[1]{\delta_{#1}}

\title{Codimension and pseudometric in co-Heyting algebras}
\author{
\begin{minipage}{6cm} 
Luck Darni\`ere \\ 
\begin{small}
D\'epartement de math\'ematiques \\[-1ex] 
Facult\'e des sciences \\[-1ex] 
Universit\'e d'Angers, France
\end{small}
\end{minipage} 
\begin{minipage}{6cm} 
Markus Junker \\
\begin{small}
Mathematisches Institut \\[-1ex] 
Abteilung f\"ur mathematische Logik \\[-1ex] 
Universit\"at Freiburg, Deutschland
\end{small}
\end{minipage}
\bigskip
}

\date{December 8, 2008}

\begin{document}

\maketitle

\begin{abstract}\scriptsize
In this paper we introduce a notion of dimension and codimension for
every element of a distributive bounded lattice $L$. These notions
prove to have a good behavior when $L$ is a co-Heyting algebra. In
this case the codimension gives rise to a pseudometric on $L$ which
satisfies the ultrametric triangle inequality. We prove that the
Hausdorff completion of $L$ with respect to this pseudometric is
precisely the projective limit of all its finite dimensional
quotients. This completion has some familiar metric properties, such as
the convergence of every monotonic sequence in a compact subset. It
coincides with the profinite completion of $L$ if and only if it is
compact or equivalently if every finite dimensional quotient of $L$ is
finite. In this case we say that $L$ is precompact. If $L$ is
precompact and Hausdorff, it inherits many of the remarkable
properties of its completion, specially those regarding the join/meet
irreducible elements. Since every finitely presented co-Heyting
algebra is precompact Hausdorff, all the results we prove on the
algebraic structure of the latter apply in particular to the former.
As an application, we obtain the existence for every positive integers
$n,d$ of a term $t_{n,d}$ such that in every co-Heyting algebra
generated by an $n$\--tuple $a$, $t_{n,d}(a)$ is precisely the maximal
element of codimension $d$.
\end{abstract}

\footnotestar{{\bf MSC 2000:} 06D20, 06B23, 06B30, 06D50}

\section{Introduction}%
\label{se:dim-and-codim}

We attach to every element of a distributive bounded lattice $L$ a
(possibly infinite) dimension and codimension, by copying 
analogous definitions in algebraic geometry. The definitions are
second order, in terms of chains of prime filters of $L$ ordered by
inclusion, but yield geometric intuition on the elements of $L$. In
the meantime we introduce a first order notion of rank and corank for
the elements of $L$. When the dual of $L$ (that is the same lattice with the reverse order)
is a Heyting algebra, we prove in section~\ref{se:definissabilite}
that the rank and dimension coincide, as well as the finite corank and
finite codimension. This ensures a much better behaviour for the
dimension and codimension (and for the rank and corank) than in
general lattices. Hence we restrict ourselves to this class, known as the
variety of {\df co-Heyting algebras} or {\df Brouwerian lattices}. 

By defining the dimension of $L$ itself as the dimension of its
greatest element, the connection is made with the so-called
``slices'' of Heyting algebras, studied by Hosoi \cite{hoso-1967},
Komori \cite{komo-1975} and Kuznetsov \cite{kuzn-1975}, among others.
More precisely, a co-Heyting algebra $L$ has dimension $d$ if and only
if its dual belongs to the $(d+1)$\--th slice of Hosoi. On the other
hand, the (co)dimension of an {\em element of} $L$ seems to be a new
concept in this area. 

In section~\ref{se:pseudometric} we introduce a pseudometric on
co-Heyting algebras based on the codimension, but delay until
section~\ref{se:completion} the study of complete co-Heyting
algebras. By elementary use of Kripke models and the finite model
property of intuitionistic propositional calculus, we check in
section~\ref{se:codim-HA-libre} that the filtration by finite
codimensions has several nice properties in any finitely generated
co-Heyting algebra $L$:
\begin{enumerate}
  \item\label{it:def-tnd}%
    For every positive integer $d$, the set $dL$ of elements of $L$ of
    codimension $\geq d$ is a principal ideal. 
  \item\label{it:def-precomp}%
    For every positive integer $d$, the quotient $L/dL$ is finite. 
  \item\label{it:dec-hausdorff}%
    If moreover $L$ is finitely presented, then
    $\displaystyle\bigcap_{d<\omega}dL=\{\ZERO\}$.
\end{enumerate}

Property (\ref{it:dec-hausdorff}) asserts that $L$ is Hausdorff (with
respect to the topology of the pseudometric we introduce). Property
(\ref{it:def-precomp}) shows that $L$ is precompact (in the sense that
its Hausdorff completion is compact). More generally we prove that a
variety $\cV$ of co-Heyting algebras has the finite model property if
and only if every algebra free in $\cV$ is Hausdorff. In such a
variety we have the following relations:

\begin{displaymath}
  \begin{array}{rclcl}
  \mbox{finitely generated} &\Longrightarrow& \mbox{precompact} & & \\
  \mbox{finitely presented} &\Longrightarrow& \mbox{precompact Hausdorff} 
                            &\Longrightarrow& \mbox{residually finite} \\
\end{array}
\end{displaymath}

Many algebraic properties probably known for finitely presented co-Heyting
algebras (but hard to find in the literature) generalize to precompact
Hausdorff co-Heyting algebras, as we show in
section~\ref{se:pre-compact}. We prove in particular that $L$ and its
completion have the same join irreducible elements, that all of them
are completely join irreducible and that every element $a \in L$ is the
complete join of its join irreducible components (the maximal join
irreducible elements smaller than $a$). We prove similar (but not
completely identical) results for the completely meet irreducible
elements. A characterisation of meet irreducible elements which are
not completely meet irreducible is also given. 

Finally we prove in section~\ref{se:completion} that the Hausdorff
completion of every co-Heyting algebra $L$ is also its
pro-finite-dimensional completion, that is the projective limit of all
its finite dimensional quotients. This completion has some nice metric
properties, such as the convergence of every monotonic sequence in a
compact subset. It coincides with the profinite completion of $L$ if
and only if it is compact or equivalently if every finite dimensional
quotient of $L$ is finite. 

So in the Hausdorff precompact case, our completion is nothing but the
classical profinite completion studied in
\cite{bezh-gehr-mine-mora-2006}. But there is an important difference:
in our situation every precompact co-Heyting algebra inherits many of
the nice properties of its completion, while in general the
properties of profinite co-Heyting algebras do not pass to their
dense subalgebras (which are exactly all residually finite
co-Heyting algebras, a much wider class than the class of precompact
Hausdorff ones). 

In the appendix we derive from (\ref{it:def-tnd}) a surprising
application: for all positive integers $n,d$ there exists a term
$t_{n,d}(x)$ with $n$ variables such that if $L$ is any co-Heyting
algebra generated by a tuple $a \in L^n$ then $t_{n,d}(a)$ is the
generator of $dL$. Possible connections with locally finite varieties
of co-Heyting algebras are discussed.

\begin{remark}\label{re:related-paper}
  The results of section~\ref{se:pre-compact} on precompact co-Heyting
  algebras are closely related to those that we derived in
  \cite{darn-junk-2008} from Bellissima's construction of a Kripke
  model for each finitely generated free Heyting algebra. Actually the
  approaches that we have developed here and in \cite{darn-junk-2008} 
  are quite complementary. The general methods of the
  present paper do not seem to be helpful for certain results, which
  are proper to finitely generated co-Heyting algebras (in particular
  those which concern the generators). On the other hand they allow us
  to recover with simple proofs many of the remarkable algebraic
  properties of finitely presented co-Heyting algebras, widely
  generalised to precompact Hausdorff co-Heyting algebras, without
  requiring any sophisticated tool of universal algebra nor the
  intricate construction of Bellissima. 
\end{remark}

\begin{remark}\label{re:pourquoi-dual}
  The reader accustomed to Heyting algebras will certainly find very
  annoying to reverse by dualisation all his/her habits. We apologise
  for this, but there were pretty good reasons for doing so. Indeed we
  have not invented the (co)dimension: we simply borrowed it from
  algebraic geometry {\it via} the Stone-Priestley duality (see
  example~\ref{ex:treillis-geom}). So we could not define in a
  different way the (co)codimension for the elements of a general
  lattice. Then it turns out that only in co-Heyting algebras we were
  able to prove that the codimension and the corank coincide when they
  are finite. Since all the results of this papers require the basic
  properties that we derive from this coincidence, we had actually no
  other choice than to focus on these algebras. 
\end{remark}

\paragraph{Acknowledgement.} 
The authors warmly thank Guram Bezhanishvili, from the New Mexico
state university, for the numerous accurate remarks and valuable
comments that he made on a preliminary version of this paper. 
Main parts of this work were done when the second author was invited
professor at the university of Angers in July 2005, and when the
first author was invited at the institute of Mathematics of Freiburg
in July 2008.

\section{Prerequisites}
\label{se:preprequesites}

\paragraph{Distributive bounded lattice.}
The language of distributive bounded lattices is
$\llat=\{\ZERO,\UN,\join,\meet\}$, the order being defined by $a \leq b$ iff 
$a=a \meet b$. We will denote by $\jjoin$ the join and by $\mmeet$ the meet of
any family of elements of a lattice. We write
$\conj$, $\disj$ for the logical connectives `and', `or' and $\cconj$,
$\ddisj$ for their iterated forms. 

We refer the reader to any book on lattices for the notions of
(prime) ideals and (prime) filter of $L$. We denote by $\Spec L$ the
{\df prime filter spectrum}, that is the set of all prime
filters of $L$. For every $a$ in $L$ let:
\begin{displaymath}
  F(a) = \{ \p \in \Spec L  \tq a \in \p\}
\end{displaymath}
As $a$ ranges over $L$ the family of all the $F(a)$'s forms a basis of
closed sets for the {\df Zariski topology} on $\Spec L$. It also forms
a lattice of subsets of $\Spec L$ which is isomorphic to $L$
(Stone-Priestley duality).

\paragraph{Dualizing ordered sets.}
An ordered set is a pair $(E, \leq$) where $E$ is a set and $\leq$ a
reflexive, symmetric and transitive binary relation. We do not require
the order to be linear. For every $x \in E$ we denote:
\begin{eqnarray*}
  x\uparrow  = \{y \in E \tq x \leq y\}, \quad x\downarrow  = \{y \in E \tq y \leq x\}, \\
  x\Uparrow  = \{y \in E \tq x < y\}, \quad x\Downarrow  = \{y \in E \tq y < x\}.
\end{eqnarray*}
The {\df dual of $E$}, in notation $E^*$, is simply the set $E$ {\em with
the reverse order}. For any $x \in E$ we will denote by $x^*$ the element
$x$ itself {\em seen as an element of} $E^*$, so that:
\begin{displaymath}
  y^* \leq x^* \iff x \leq y 
\end{displaymath}
The stars indicate that the first symbol $\leq$ refers to the order of
$E^*$, while the second one refers to the order of $E$. Similarly
$X^*=\{x^*\tq x\in X\}$ for every $X \subseteq E$ hence for instance
$x\downarrow=(x^*\uparrow)^*$.  

This apparently odd notation is specially convenient when $E$ carries
an additional structure. For example the dual $L^*$
of a distributive bounded lattice $L$ is obviously a distributive
bounded lattice and for every $a,b \in L$:
\begin{eqnarray*}
  \ZERO^* = \UN         &\hbox{and}& \UN^* =\ZERO\\
  (a \join b)^* = a^* \meet b^* &\hbox{and}& (a \meet b)^* = a^* \join b^*
\end{eqnarray*}

\paragraph{(Co)foundation rank and ordered sets.}
The appropriate generalisations to arbitrary ordinals of ``the length
of the longest chain'' of elements in $(E,\leq)$ are the foundation rank
and cofoundation rank of an element $x$ of $E$. The {\df foundation
rank} is inductively defined as follows:
\begin{displaymath}
  \begin{array}{lcl}
    \rk x \geq 0,           &      & \\
    \rk x \geq \alpha=\beta+1        & \iff & \exists y \in E,\ x>y\hbox{ and }\rk y \geq \beta, \\
    \displaystyle
    \rk x \geq \alpha=\bigcup_{\beta<\alpha} \beta & \iff & \forall \beta < \alpha,\ \rk x \geq \beta.
    \end{array}
\end{displaymath}
If there exists an ordinal $\alpha$ such that $\rk x \geq \alpha$ and $\rk x \ngeq \alpha+1$
then $\rk x = \alpha$ otherwise $\rk x = +\infty$. The {\df cofoundation rank}
is the foundation rank with respect to the reverse order, that is:
\begin{displaymath}
  \forall x \in E,\ \cork x = \rk x^* 
\end{displaymath}

\paragraph{(Co)dimension and lattices.}
For every element $a$ and every prime filter $\p$ of a distributive
bounded lattice $L$ we let:
\begin{itemize}
  \item $\hgt\p =$ the foundation rank of $\p$ in $\Spec L$
    (ordered by inclusion)
  \item $\cohgt\p =$ the cofoundation rank of $\p$ in $\Spec L$
  \item $\dim_L a=\sup\{\cohgt\p\tq\p \in \Spec L,\ a \in \p\}$
  \item $\codim_L a=\min\{\hgt\p\tq\p \in \Spec L,\ a \in \p\}$
\end{itemize}
Here we use the convention that the supremum (resp. minimum) of an
empty set of ordinals is $-\infty$ (resp. $+\infty$). Hence $\ZERO$ has
codimension $+\infty$ and is the only element of $L$ with dimension $-\infty$.
The subscript $L$ is omitted whenever it is clear from the context.

\begin{remark}\label{re:codim-join-max}%
  The following fundamental (and intuitive) identities
  follow immediately from the above definitions, and the fact that
  $F(a \join b)=F(a)\cup F(b)$:
\begin{eqnarray*} 
    \dim a \join b &=& \max(\dim a,\dim b) \\
  \codim a \join b &=& \min(\codim a,\codim b)
\end{eqnarray*}
\end{remark}

Finally we define the {\df dimension of the lattice} $L$, in
notation $\dim L$, as the dimension of $\UN_L$. Observe that:
\begin{eqnarray*}
  \dim L \leq d & \iff & \forall a \in L,\enskip \dim a \leq d \\
             & \iff & \forall a \in L\setminus\{\ZERO\},\enskip \codim a \leq d
\end{eqnarray*}

\begin{example}\label{ex:treillis-geom} 
  Consider the lattice\footnote{Note that this is the lattice of all
  Zariski closed subsets of $k^n$ hence a co-Heyting algebra, not a
  Heyting algebra.} $L(k^n)$ of all algebraic varieties in the affine
  $n$-space over an algebraically closed field $k$. The prime filter
  spectrum of $L(k^n)$ is homeomorphic to the usual spectrum of the
  ring $k[X_1,\dots,X_n]$. For any algebraic variety $V \subseteq k^n$, the
  (co)dimension of $V$ as an element of $L(k^n)$ is nothing but its
  geometric (co)dimension, that algebraic geometers define in terms of
  length of chains in $\Spec k[X_1,\dots,X_n]$. In particular $\dim L(k^n)
  = \dim k^n = n$.
\end{example}

\paragraph{(Co)rank and the strong order.}
For every $a,b$ in a distributive bounded lattice $L$ we let $b \ll a$
if and only if $F(b)$ is ``much smaller'' than $F(a)$, in the sense
that $F(b)$ is contained in $F(a)$ and has empty interior inside
$F(a)$ (with other words $F(a)\setminus F(b)$ is dense in $F(a)$). This is a
definable relation in $L$:
\begin{displaymath}
  b \ll a \iff \forall c,\ (a \leq b \join c \Rightarrow a \leq c) 
\end{displaymath}
This is a strict order on $L\setminus\{\ZERO\}$ (but not on $L$ because
$\ZERO\ll\ZERO$). Nevertheless we call it the {\df strong order} on $L$.
Obviously $b \ll a$ implies that $b<a$ whenever $a$ or $b$ is non-zero.
From now on, except if otherwise specified, when we will speak of the
{\df rank} and {\df corank} of an element $a$ of $L\setminus\{\ZERO\}$, in
notation $\rk_L a$ and $\cork_L a$, we will refer to the foundation
rank and cofoundation rank of $a$ in $L\setminus\{\ZERO\}$ {\em with respect to
the strong order} $\ll$. As usually the subscript $L$ will often be
omitted.

\paragraph{Co-Heyting algebras.} 
Let $\ltc=\llat \cup \{-\}$ be the language of co-Heyting algebras and
$\lha={\llat \cup \{\to\}}$ the language of Heyting algebras. The additional
operations are defined by: 
\begin{displaymath}
  a-b = (b^*\to a^*)^* = \min\{c\tq a \leq b \join c\} 
\end{displaymath}
So the strong order is {\em quantifier-free} definable in co-Heyting
algebras: 
\begin{displaymath}
  b \ll a \quad\iff\quad b \leq a = a-b
\end{displaymath}
Either by dualizing known results on Heyting algebras or by
straightforward calculation using Stone-Priestley duality (see
footnote~\ref{fo:stone-priestley-HA}) the following rules are easily
seen to be valid in every co-Heyting algebra: 
\begin{itemize}
  \item
    $a = (a-b) \join (a \meet b)$.
  \item
    $(a \join b) - c = (a-c) \join (b-c)$.
  \item
    $a - (b \join c) = (a-b)-c$.
  \item
    $a-(a-b)=(a \meet b)-(a-b)$.
  \item 
    $(a-b) \meet b \ll a$.
\end{itemize}
Note in particular that $b \leq a$ if and only if $b-a = \ZERO$, and that
$a-(a-b) \leq b$. We will use these rules in several calculations without
further mention. 
\smallskip

In a co-Heyting algebra $L$ we denote by $a \bigtriangleup b$ the 
{\df topological symmetric difference}%
\footnote{\label{fo:stone-priestley-HA}%
Note that $F(a-b)$ is the topological closure of $F(a) \setminus F(b)$ in
$\Spec L$. So $F(a \bigtriangleup b)$ is the topological closure of the usual
symmetric difference $(F(a)\setminus F(b))\cup(F(b)\setminus F(a))$.}:
\begin{displaymath} 
  a \bigtriangleup b = (a-b) \join (b-a) = (a^* \leftrightarrow b^*)^*
\end{displaymath} 
This is a commutative, non-associative operation. 
Note that $a \bigtriangleup b = \ZERO$ if and only if $a=b$. Moreover the following
``triangle inequality'' for $\bigtriangleup$ will be useful:
\begin{displaymath}
  a \bigtriangleup c \leq (a \bigtriangleup b) \join (b \bigtriangleup c)
\end{displaymath}

We remind the reader (dualizing basic properties of Heyting
algebras) that each ideal $I$ of $L$ defines a congruence $\equiv_I$ on $L$:
\begin{displaymath}
  a \equiv_I b \iff a \bigtriangleup b \in I 
\end{displaymath}
So the quotient $L/I$ carries a natural structure of co-Heyting
algebra which makes the canonical projection $\pi_I:L \to L/I$ an
$\ltc$\--morphism.

Conversely every congruence $\equiv$ on $L$ is of that kind. Indeed
$I_\equiv=\{a\in L \tq a \equiv \ZERO \}$ is an ideal of $L$ and $\equiv_{I_\equiv}$ is
precisely $\equiv$.

The {\df kernel} $\Ker f=f^{-1}(\{\ZERO\})$ of any morphism $f:L\to L'$ of
 co-Heyting algebra is an ideal of $L$. Given an ideal $I$ of $L$
there is a unique morphism $g:L/I \to L'$ such that $f=g \circ \pi_I$ if and
only if $I \subseteq \Ker f$. If $f$ is onto, then so is $g$. If moreover
$I=\Ker f$ then $g$ is an isomorphism and we will identify $L/I$ with
$L'$ and $f$ with $\pi_I$.

For every ordinal $d$ we set:
\begin{displaymath}
  dL = \{ a \in L \tq \codim a \geq d \}
\end{displaymath}
By remark~\ref{re:codim-join-max} this is an ideal of $L$. The {\df
generator of $dL$}, whenever it exists, will be denoted $\varepsilon_d(L)$. The
canonical projection $\pi_{dL}:L \to L/dL$
will simply be denoted $\pi_d$ when the context makes it unambiguous. 

\begin{remark}\label{re:identification-quotient}
  Given a {\em surjective} $\ltc$\--morphism $\varphi:L \to L'$, if
  $\varphi^{-1}(dL')=dL$ then there exists a unique isomorphism 
  $d\varphi:L/dL \to L'/dL'$ such that $\pi_{dL'} \circ \varphi = d\varphi \circ \pi_{dL}$. 
  In this situation we will identify $L'/dL'$ with $L/dL$ and say
  that:
\begin{displaymath}
  \varphi^{-1}(dL')=dL\quad \Longrightarrow \quad 
  L'/dL' = L/dL \hbox{ \ and \ } \pi_{dL'} \circ \varphi = \pi_{dL}
\end{displaymath}
\end{remark}

\paragraph{Pseudometric spaces.}

A map $\dist{}:X \times X \to \RR$
such that for every $x,y,z$ in $X$, $\dist{}(x,x)=0$,
$\dist{}(x,y)=\dist{}(y,x)\geq 0$ and $\dist{}(x,z) \leq \dist{}(x,y) +
\dist{}(y,z)$ (triangle inequality), is called a {\df pseudometric} on
the set $X$. It is a {\df metric} if and only if
moreover $\dist{}(x,y)\neq0$ whenever $x\neq y$. 
For example, if $(Y,\dist{Y})$ is a metric space and $f:X \to Y$ a
surjective map then $\dist{Y}(f(x),f(y))$ defines a pseudometric on
$X$. Every pseudometric on $X$ is of that kind. Indeed $\dist{}$
induces a metric $\dist{}'$ on the quotient $X'$ of $X$ by the
equivalence relation:
\begin{displaymath}
  x \sim y \iff \dist{}(x,y)=0   
\end{displaymath}
Lipschitzian maps between pseudometric spaces are defined as in the
metric case. So are the open balls and the topology determined by a
pseudometric. Lipschitzian functions are obviously continuous. Note
also that a pseudometric is a metric if and only its topology is
Hausdorff. So $X/\sim$ defined above is the largest Hausdorff quotient of
$X$. 

The {\df Hausdorff completion} of a pseudometric space $X$ is a complete
{\em metric} space $X'$ together with a continuous map $\iota_X:X \to X'$
such that $\iota_X(X)$ is dense in $X'$, and for every continuous map $f$
from $X$ to a complete metric space $X''$ there is a unique continuous
map $g:X'\to X''$ such that $f=g \circ \iota_X$. Note that if $f$ is
$\lambda$-Lipschitzian then so is $g$. The Hausdorff completion of $X$,
which is unique up to isomorphism by the above universal property, is
also the completion of the largest Hausdorff quotient of $X$.

\section{Axiomatization}%
\label{se:definissabilite}

In this section we prove that the (co)dimension and (co)rank coincide,
at least when they are finite, in every co\--Heyting algebra. One can
show that this in not true in every distributive bounded lattices.
Only the inequalities of proposition~\ref{pr:dim-plus-grand-que-rang}
below are completely general.

\begin{example}
  Even in co-Heyting algebras non finite codimensions and coranks do 
  not coincide in general. Here is a counter-example:
  \begin{displaymath}
    \ZERO<x_\omega<\cdots<x_2<x_1<x_0=\UN
  \end{displaymath}
  Since this is a chain, it is a co-Heyting algebra $L$ in which $\ll$
  coincides with $<$ on $L \setminus \{\ZERO\}$, hence $\cork x_\alpha=\alpha$ for every $\alpha\leq\omega$.
  On the other hand each element $x_\alpha$ generates a prime filter
  $\p_\alpha$. There is only one more prime filter which is $\p=\{x_\alpha\}_{\alpha<\omega}$. 
  Clearly $\hgt\p=\omega$ hence $\hgt\p_\omega=\omega+1$. It follows that: 
  \begin{displaymath}
    \codim x_\omega = \omega+1 > \cork x_\omega 
  \end{displaymath}
\end{example}

In this section we will make extensive use of the
following facts, proved for example in \cite{hoch-1969}, theorem~1
and its first corollary. A subset of $\Spec L$ which is a boolean
combination of basic closed sets $(F(a))_{a \in L}$ is called a {\df
constructible set} (a patch in \cite{hoch-1969}).  They form a basis
of open sets for another topology on $\Spec L$ usually called the {\df
constructible topology}. Recall that a topological space $X$ is
compact if and only if every open cover has a finite subcover.

\begin{fact}\label{fa:top-constructible}
  $\Spec L$ is compact with respect to the constructible topology.
  Consequently every constructible subset of $\Spec L$ is compact with
  respect to this topology since it is closed in $\Spec L$.
\end{fact}
\begin{fact}\label{fa:cloture-constructible}
  If a prime filter $\p$ belongs to the closure (with respect to the
  Zariski topology) of a constructible subset $S$ of $\Spec L$ then
  it belongs to the closure of a point of $S$, that is $\p\supseteq\q$ for
  some $\q \in S$.
\end{fact}

\begin{proposition}\label{pr:dim-plus-grand-que-rang}
  For every nonzero element $a$ in a distributive bounded lattice: 
  \begin{displaymath}
    \dim a \geq \rk a \hbox{\quad and\quad}\codim a \geq \cork a 
  \end{displaymath}
\end{proposition}

\begin{proof}
By induction on the ordinal $\alpha$ we prove that if $\crk a \geq \alpha$ then
$\cdim a \geq \alpha$. This is trivial if $\alpha=0$ because $a\neq\ZERO$.

Assume $\alpha = \beta+1$. Let $b \ll a$ in $L \setminus \{\ZERO\}$ be such that $\rk b \geq
\beta$. The induction hypothesis gives a prime filter $\q$ of coheight at
least $\beta$ containing $b$. Then $\q$ also contains $a$, and since $b \ll
a$, $\q$ belongs to the Zariski closure of $F(a) \setminus F(b)$. It follows that $\q
\supset \p$ for some $\p$ in $F(a) \setminus F(b)$ by
fact~\ref{fa:cloture-constructible}. Then $\cohgt\p \geq \beta+1$, hence
$\dim a \geq \beta+1$.

Assume $\alpha$ is a limit ordinal. For every $\beta<\alpha$, $\rk a \geq \beta$ hence
$\dim a \geq \beta$ by the induction hypothesis, so $\dim a \geq \alpha$. 

\smallskip

We turn now to the codimension. Let $b$ in $L \setminus \{\ZERO\}$ be such that
$\cork b \geq \alpha$. Assume that
$\alpha = \beta+1$ and let $a$ in $L$ be such that $b \ll a$ and $\cork a \geq \beta$.
Choose any prime filter $\q$ containing $b$. Then $\q$ also contains
$a-b$ because $b < a = a-b$. So $\q$ belongs to the closure of 
$F(a) \setminus F(b)$ hence to the closure of some $\p$ in $F(a) \setminus F(b)$ by
fact~\ref{fa:cloture-constructible}. By induction hypothesis 
$\codim a \geq \beta$ hence $\hgt\p \geq \beta$ and thus $\hgt\q \geq \beta + 1$.
Since this is true for every $\q \in F(b)$ it follows that 
$\codim b \geq \beta+1$.

The limit case is as above.
\end{proof}

\begin{proposition}\label{pr:dim-egale-rang}
  In co-Heyting algebras the dimension coincides with the foundation
  rank with respect to $\ll$ for every nonzero element.
\end{proposition}

\begin{proof}
It suffices to prove, by induction on the ordinal $\alpha$, that if $\dim a
\geq \alpha$ then $\rk a \geq \alpha$. This is obvious if $\alpha=0$ since $a \neq \ZERO$. The
limit case is clear as well.

Assume that $\alpha=\beta+1$, let $\p$ be a prime filter of coheight
at least $\beta+1$ containing $a$. Let $\q \supset \p$ be a prime filter
of coheight $\beta$ and $a'$ an element of $\q \setminus \p$. Then $\p$ 
belongs to $F(a) \setminus F(a')$, hence to $F(a-a')$. In other words $a-a'$
belongs to $\p$, hence to $\q$. Let $b=a' \meet (a-a')$, then $b \ll a$ by
construction. Moreover $b \in \q$ hence $\dim b \geq \beta$. By induction
hypothesis it follows that $\rk b \geq \beta$, hence $\rk a \geq \beta+1$.
\end{proof}

For every element $a$ in a distributive bounded lattice $L$ let $mF(a)$
denote the set of minimal elements of $F(a)$, that is the prime filters
which are minimal with respect to the inclusion among those containing
$a$.

\begin{lemma}\label{le:min-compact}
  Let $L$ be a co-Heyting algebra and $a,b\in L$.
  \begin{displaymath}
    mF(a-b)=mF(a) \cap F(b)^c = mF(a) \cap F(a-b)
  \end{displaymath}
  So the Zariski and the constructible topologies induce the same
  topology on $mF(A)$. It follows that $mF(a)$ is a Boolean space, and
  in particular a compact space.
\end{lemma}

\begin{proof}
The two last statements follow immediately from the second
equality, so let us prove these two equalities.

We already mentioned in footnote~\ref{fo:stone-priestley-HA} 
that $F(a-b)=\overline{F(a)\setminus F(b)}$, where the
line stands for the Zariski closure in $\Spec L$. The set of minimal 
elements of $F(a)\setminus F(b)$ is clearly $mF(a)\setminus F(b)$. So by
fact~\ref{fa:cloture-constructible}:
\begin{displaymath}
   mF(a-b)=mF(a) \setminus F(b)
\end{displaymath}
This proves the first equality. It implies that $mF(a-b)\subseteq mF(a)$
hence:
\begin{displaymath}
  mF(a-b) \subseteq mF(a) \cap F(a-b)  
\end{displaymath}
Conversely every element of $F(a-b)$ which is minimal in $F(a)$ is {\it
a fortiori} minimal in $F(a-b)$ because $F(a-b)\subseteq F(a)$. So the second
equality is proved.
\end{proof}

\begin{proposition}\label{pr:codim-egale-corang}
  Let $a$ be any nonzero element of a co-Heyting algebra $L$, let $\p$ a
  prime filter of $L$ and $n$ a positive integer. 
  \begin{enumerate}
    \item If $\hgt\p \geq n$ then $\p$ contains an element of codimension at
      least $n$. 
    \item If $\codim a \geq n$ then $\cork a \geq n$.
  \end{enumerate}
\end{proposition}

\begin{proof}
If $n=0$ the first statement is trivial. Assume that it has been 
proved for $n-1$ with $n\geq 1$. Let $\p' \subset \p$ be such that 
$\hgt \p' \geq n-1$. The induction hypothesis gives $b\in\p'$ such that
$\codim b \geq n-1$. For any $\q \in mF(b)$, $\p \not\subseteq \q$ so 
we can choose $b_\q \in \p \setminus \q$. The intersection of all the
$F(b_\q)$'s with $mF(b)$ is empty. By lemma~\ref{le:min-compact},
$mF(b)$ is compact hence there exists a finite subfamily
$(F(b_{\q_i}))_{i \leq r}$ whose intersection with $mF(b)$ is empty. Let:
\begin{displaymath}
  a = b \meet \mmeet_{i \leq r}b_{\q_i} 
\end{displaymath}
By construction $a \in \p$, $a<b$ and:
\begin{displaymath}
  mF(b)\setminus F(a) \supseteq mF(b) \setminus \bigcap_{i\leq r}F(b_{\q_i}) = mF(b)
\end{displaymath}
So $mF(b)\setminus F(a)=mF(b)$, but $mF(b-a)=mF(b)\setminus F(a)$
by lemma~\ref{le:min-compact}, so we have
proved that $mF(b-a)=mF(b)$. Hence $F(b-a)=F(a)$ by
fact~\ref{fa:cloture-constructible}, that is $b-a=b$. 
It follows that $a \ll b$ hence $\cork b \geq \cork a + 1 = n$.

\smallskip

The second statement is trivial as well if 
$n=0$. So let us assume that $n\geq1$ and the result is proved for $n-1$. 
For every $\p \in mF(a)$, $\hgt\p\geq n$ so we can choose a prime filter
$\q \subset \p$ such that $\hgt\q \geq n-1$. The previous point then gives 
an element $a_\p\in\q$ such that $\codim a_\p \geq n-1$. By construction
$a\notin\q$ because $\p$ is minimal in $F(a)$, hence $a_\p-a\in\q$ and {\it a
fortiori} $a_\p-a\in\p$. So $mF(a)$ is covered by 
$\big(F(a_\p-a)\big)_{\p\in mF(a)}$.
This is an open cover for the constructible topology, and $mF(a)$ is
compact for this topology by lemma~\ref{le:min-compact}, so there is a
finite subfamily $(F(a_{\p_i}-a))_{i\leq r}$ which covers $mF(a)$. Let
$b=\jjoin_{i\leq r}(a_{\p_i}-a)$. By construction $mF(a)$ is contained in $F(b)$
hence $a \leq b$, and moreover: 
\begin{displaymath}
  b-a = \jjoin_{i \leq r} (a_{\p_i}-a)-a 
      = \jjoin_{i \leq r}  a_{\p_i}-a = b
\end{displaymath}
That is $a \ll b$, so $\cork a \geq \cork b +1$. Finally:
\begin{displaymath}
  \codim b = \min_{i \leq r} \codim (a_{\p_i}-a) 
           \geq \min_{i \leq r} \codim a_{\p_i} \geq n-1
\end{displaymath}
By induction hypothesis it follows that $\cork b \geq n-1$, hence $\cork
a \geq n$. 
\end{proof}

Once put together, propositions~\ref{pr:dim-plus-grand-que-rang},
\ref{pr:dim-egale-rang} and \ref{pr:codim-egale-corang} imply that
$\dim a = \rk a$, and that $\codim a = \cork a$ whenever $\cork a$ is
finite, for every non zero element $a$ in a co-Heyting algebra $L$.
This result is the corner stone of this paper. Indeed the
(co)dimension has geometrically intuitive properties
(remark~\ref{re:codim-join-max}) that the (co)rank seems to be
lacking. On the other hand the definition of the (co)dimension is not
first-order, while the (co)rank is defined only in terms of the strong
order which is first order definable. When both coincide the best of
the two notions can be put together. Let us emphasize this coincidence.

\begin{theorem}\label{th:dim-et-rang}
  For every co-Heyting algebra $L$, every element $a$ of $L$ and every
  positive integer $d$:
  \begin{eqnarray*}
    \dim a \geq d &\iff& \exists x_0,\dots,x_d,\ \ZERO\neq x_d \ll\cdots\ll x_0\leq a \\
    \codim a \geq d &\iff& \exists x_0,\dots,x_d,\ a\leq x_d \ll\cdots\ll x_0 
  \end{eqnarray*}
  In particular\,\footnote{Recall that we defined $dL=\{a \in L\,/\, \codim
  a \geq d\}$.} $dL$ is uniformly definable by a positive existential
  $\ltc$\--formula.
\end{theorem}

\begin{proof}
The two equivalences have already been proved. The last statement
follows since $\ll$ is definable by a positive quantifier free
formula: $b \ll a$ iff $b \leq a \conj a-b=a$. 
\end{proof}

\begin{corollary}\label{co:TC-morph-codim}%
  Let $\varphi:L \to L'$ be an $\ltc$\--morphism and $d$ a positive integer.
  \begin{enumerate}
    \item\label{it:morph-lipshitz}%
      $\varphi(dL) \subseteq dL'$.
    \item\label{it:phi-de-dL}%
      If $\varphi$ is surjective then: 
      \begin{enumerate}
        \item $\varphi(dL) = dL'$.
        \item\label{it:dL-Ker-phi}%
          $\dim L' < d \iff dL \subseteq \Ker \varphi$
        \item 
          $\Ker \varphi \subseteq dL \iff \varphi^{-1}(dL')=dL$.
      \end{enumerate}
    \item\label{it:phi-de-epsilon}%
      If $\varphi$ is surjective and $dL$ is principal then $dL'$
      is principal and $\varphi(\varepsilon_d(L))=\varepsilon_d(L')$.
  \end{enumerate}
\end{corollary}

\begin{proof}
(\ref{it:morph-lipshitz})\enskip 
By theorem~\ref{th:dim-et-rang}, $dL$ and 
$dL'$ are both defined by the same positive existential 
$\ltc$\--formula, hence $\varphi(dL) \subseteq dL'$. 
\smallskip

(\ref{it:phi-de-dL})\enskip 
For the first point it is sufficient to check that
$dL' \subseteq \varphi(dL)$. If $d=0$ then $dL=L$ and 
$\varphi(L)=L'$ because $\varphi$ is surjective. Now assume that $d \geq 1$. 
For any $b' \in dL' \setminus \{\ZERO\}$ theorem~\ref{th:dim-et-rang} gives $a' \in L'$ such
that $b' \ll a'$ and $\codim a' \geq d-1$. Let $a,b \in L$ be such that $\varphi(a)=a'$ 
and $\varphi(b)=b'$. By induction hypothesis $a$ can be chosen in $(d-1)L$. 
Lastly let $x=a-b$ and $y=x \meet b$. Then:

$\varphi(x)=\varphi(a)-\varphi(b)=a'-b'=a'$. 

$\varphi(y) = \varphi(x) \meet \varphi(b) = a' \meet b' = b'$.

$x-y = x - (x \meet b) = x-b = (a-b)-b = a-b =x$. 

Moreover $y \leq x$ so $y \ll x$. Note that $x,y \neq \ZERO$ since their
respective images are non zero. By theorem~\ref{th:dim-et-rang} again
it follows that $\codim y > \codim x \geq \codim a$ hence $y \in dL$. 

Equivalence (\ref{it:dL-Ker-phi}) follows since 
$\dim L' < d$ if and only if every non zero element of $L'$ has
codimension at most $d-1$, that is $dL'=\{\ZERO\}$. 

Finally $dL=\varphi^{-1}(\varphi(dL))$ if and only if $\Ker \varphi \subseteq dL$. 
But $\varphi(dL)=dL'$ so we are done.
\smallskip

(\ref{it:phi-de-epsilon})\enskip
We already know that $\varphi(\varepsilon_d(L))\in dL'$. For any $a' \in dL'$ let $a \in dL$
such that $\varphi(a)=a'$. Then $a \leq \varepsilon_d(L)$ hence $a' \leq \varphi(\varepsilon_d(L))$.
\end{proof}

\begin{remark}\label{re:dim-L/dL}%
  Corollary~\ref{co:TC-morph-codim}(\ref{it:phi-de-dL}) implies that 
  $\dim L/dL <d$ for every positive integer $d$ and every co-Heyting
  algebra $L$.
\end{remark}

\begin{corollary}\label{co:epsilon-ll}
  Let $L$ be a co-Heyting algebra such that $dL$ and $(d+1)L$ are
  principal for some $d$. Then $\varepsilon_{d+1}(L) \ll \varepsilon_d(L)$.
\end{corollary}

\begin{proof}
If $\varepsilon_{d+1}(L) = \ZERO$ this is obvious. Otherwise by
theorem~\ref{th:dim-et-rang} there an element $a$ in $L$ such that
$\varepsilon_{d+1}(L) \ll a$ and $\codim a =d$. Then $a \leq \varepsilon_d(L)$ by definition
hence $\varepsilon_{d+1}(L) \ll \varepsilon_d(L)$.
\end{proof}

\phantomsection
\addcontentsline{toc}{subsection}%
  {Codimension and slices}
\subsection*{Codimension and slices}

The dimension of a co-Heyting algebra should be a familiar notion to the
specialists in Heyting algebras, since it coincides after dualisation
with the notion of ``slice'', which can be defined as follows. 
Let $P_n(x_1,\dots,x_n)$ be a term defined inductively by $P_0=\UN$ and:
\begin{displaymath}
  P_{n+1}=(P_n-x_{n+1})\meet x_{n+1}
\end{displaymath}
Let $\cS_n$ denote the variety of
co-Heyting algebras satisfying the equation $P_n=\ZERO$, and $\cS_n^*$
the corresponding variety of Heyting algebras. The variety $\cS_n^*$ 
appears for example in \cite{komo-1975}. The above axiomatization is
mentioned in \cite{bezh-2001}. 
It is folklore that a Heyting algebra $L^*$ belongs to $\cS_n^*$ if and
only if its prime filter spectrum does not contain any chain of length
$n$, or equivalently is prime ideal spectrum has this property. So dually $L$ 
belongs to $\cS_n$ if and only if its prime filter spectrum does not 
contain any chain of length $n$, that is $\dim L < n$. For lack of
a reference, we give here an elementary proof. 

\begin{proposition}
  A co-Heyting algebra $L$ has dimension $\leq d$ if and only if it
  belongs to the $\cS_{d+1}$.
\end{proposition}

\begin{proof}
We mentioned in section~\ref{se:preprequesites} that 
$(a-b)\meet b\ll b$  for every $a,b\in L$. Then for every $a_1,\dots,a_{d+1}\in L$:
\begin{displaymath}
  P_{d+1}(a_1,\dots,a_{d+1})\ll P_d(a_1,\dots,a_d)\ll \cdots \ll P_1(a_1)\ll \UN
\end{displaymath}
So if $L$ does not belong to $\cS_{d+1}$ there is a tuple $a$ in
$L^{d+1}$ such that $P_{d+1}(a)\neq\ZERO$. Then by the above property
(and theorem~\ref{th:dim-et-rang}) $\dim L=\dim_L \UN \geq d+1$. 

Conversely if $\dim L \geq d+1$ there are $b_1,\dots,b_{d+1} \in L$ such
that:
\begin{displaymath}
  \ZERO\neq b_{d+1} \ll b_d \ll \cdots \ll b_1 \ll \UN
\end{displaymath}
Then $\UN-b_1=\UN$ hence $P_1(b_1)=b_1$, and inductively
$P_{d+1}(b_1,\dots,b_{d+1})=b_{d+1}$. Since $b_{d+1}\neq\ZERO$ it follows
that $L$ does not belong to $\cS_{d+1}$.
\end{proof}

\section{Pseudometric induced by the codimension}
\label{se:pseudometric}

The ``triangle inequality'' for $\bigtriangleup$ (see
section~\ref{se:preprequesites}) and the fundamental property of the
codimension (see remark~\ref{re:codim-join-max}) prove that the
codimension defines a {\df pseudometric $\dist{L}$ on $L$} as follows:
\begin{displaymath} 
  \dist{L}(a,b)= \left\{
  \begin{array}{ll} 
    2^{-\codim a \bigtriangleup b} & \hbox{if }\codim a \bigtriangleup b < \omega, \\
    0                 & \hbox{otherwise.}
  \end{array} \right.
\end{displaymath}
As usually the index $L$ will be omitted whenever it is clear from the
context. The topology determined by this pseudometric will be called
the {\df codimetric topology}. In the remaining of this paper, every
metric or topological notion, when applied to a co-Heyting algebra,
will refer to its pseudometric, except if otherwise specified.

Note that $\omega L$ is the topological closure of $\{\ZERO\}$ and
that a basis of neighborhood for any $x \in L$ is given, as $d$ ranges
over the positive integers, by:
\begin{equation}\label{eq:Uxd}%
  U(x,d)=\{y \in L \tq x \bigtriangleup y \in dL\} 
\end{equation}
It follows that $L$ is a {\df Hausdorff co-Heyting algebra} (with
other words its codimetric topology is Hausdorff, or equivalently
$\dist{L}$ is a metric) if and only if $\omega L =\{\ZERO\}$, that is if
every non-zero element of $L$ has finite codimension. Note that 
the largest Hausdorff quotient of $L$ (as a pseudometric space) is
exactly $L/\omega L$.

A pseudometric space is called precompact if and only if its Hausdorff
completion is compact. It will be shown in section~\ref{se:completion}
that $L$ is a {\df precompact co-Heyting algebra} if and only if
$L/dL$ is finite for every positive integer $d$
(corollary~\ref{co:complet-compact}). Until then we simply take this
characterisation as a definition. 

\begin{remark}\label{re:dim-finie}
  If $L$ has finite dimension $d$ then $(d+1)L=\{\ZERO\}$ (see
  section~\ref{se:preprequesites}) hence the codimetric topology boils
  down to the discrete topology. In particular for every co-Heyting
  algebra $L$, the codimetric topology in $L/dL$ is discrete (see
  remark~\ref{re:dim-L/dL}). 
\end{remark}

\begin{proposition}\label{pr:morph-lipshitz}%
  Every $\ltc$\--morphism $\varphi:L \to L'$ is 1-Lipschitzian.
  In particular $\varphi$ is continuous.
\end{proposition}

\begin{proof}
For every positive integer $d$ such that $\dist{}(a,b) \leq 2^{-d}$ we 
have by corollary~\ref{co:TC-morph-codim}(\ref{it:morph-lipshitz}):
\begin{displaymath}
   a \bigtriangleup b \in dL \Rightarrow \varphi(a) \bigtriangleup \varphi(b) =  \varphi(a \bigtriangleup b) \in dL'
\end{displaymath}
that is $\dist{}(\varphi(a),\varphi(b)) \leq 2^{-d}$.
\end{proof}

We extend $\dist{L}$ to $L^n$ by setting:
\begin{displaymath}
  \dist{L}\big((a_1,\dots,a_n),(b_1,\dots,b_n)\big)=\max_{1\leq i\leq n} \dist{L}(a_i,b_i) 
\end{displaymath}
This is again a pseudometric on $L^n$. Clearly the
topology that it defines on $L^n$ is the product topology of the
codimetric topology of $L$. 

\begin{proposition}\label{pr:term-lipshitz}%
  The function $t:L^n \to L$ defined in the obvious way by an arbitrary 
  $\ltc$\--term $t(x)$ with $n$ variables (and parameters in $L$) 
  is 1-Lipshitzian. As a consequence if $L$ is Hausdorff then the set of solutions of 
  any system of equations (with parameters in $L$) is closed.
\end{proposition}

\begin{proof}
For every $a,b\in L^n$ and every positive integer $d$, 
$\dist{}(a,b) \leq 2^{-d}$ if and only if $\pi^n_d(a) = \pi^n_d(b)$, where 
$\pi^n_d:L^n \to (L/dL)^n$ is the product map induced by $\pi_d$ in the
obvious way. In this case: 
\begin{displaymath}
  \pi_d(t(a)) =t(\pi^n_d(a)) = t(\pi^n_d(b)) = \pi_d(t(b)) 
\end{displaymath}
Hence $\pi_d(t(a) \bigtriangleup t(b)) = \pi_d(t(a)) \bigtriangleup \pi_d(t(b)) =\ZERO$ that is 
$\dist{}(t(a),t(b)) \leq 2^{-d}$. This proves the first point. 

Now given any set $(t_i)_{i\in I}$ of $\ltc$\--terms with $n$ 
variables (and parameters in $L$):
\begin{displaymath}
  \{a \in L^n \tq \forall i \in I,\ t_i(a)=\ZERO\} = \bigcap_{i \in I}t_i^{-1}(\{\ZERO\})
\end{displaymath}
If $L$ is Hausdorff then $\{\ZERO\}$ is closed. So each
$t_i^{-1}(\{\ZERO\})$ is closed by continuity of $t_i$ hence so is their
intersection.
\end{proof}

\begin{proposition}\label{pr:quotient-codimetric}%
  The quotient of a Hausdorff co-Heyting algebra $L$ by
  an ideal $I$ is Hausdorff if and only if $I$ is closed. In
  particular the quotient of any Hausdorff co-Heyting
  algebra by a principal ideal is Hausdorff.
\end{proposition}

Note that closed ideals need not to be principal, see
example~\ref{ex:fg-hausdorff-non-fp}.

\begin{proof}
Let $\pi:L \to L/I$ denote that canonical projection. If the codimetric
topology on $L/I$ is Hausdorff then $\{\ZERO_{L/I}\}$ is closed hence 
$I = \pi^{-1}(\{\ZERO_{L/I}\})$ is closed because $\pi$ is continuous. 

Conversely if the codimetric topology on $L/I$ is not Hausdorff then
there exists a non zero element $a' \in L/I$ whose codimension is not
finite. Let $a \in L$ such that $\pi(a) = a'$. Note that $a \notin I$ because
$a' \neq \ZERO$. For every $d$,
corollary~\ref{co:TC-morph-codim}(\ref{it:phi-de-dL}) gives an $a_d \in dL$
such that $\pi(a_d)=a'$. The sequence $(a_d)_{d<\omega}$ is convergent to
$\ZERO$ hence $a_d \bigtriangleup a$ is convergent to $a$. But $a_d \bigtriangleup a \in I$ for
every $d$ since $\pi(a_d)=a'=\pi(a)$ so $I$ is not closed. 

The last statement follows since an ideal generated by a single
element $a$ is obviously closed: it is the inverse image of the closed set
$\{\ZERO\}$ by the continuous map $x \mapsto x - a$. 
\end{proof}

\section{The finitely generated case}%
\label{se:codim-HA-libre}

We prove in this section that finitely generated co-Heyting algebras
are precompact, and Hausdorff if moreover they are finitely presented.
This mostly a rephrasing of known facts. It can be derived for example
from Bellissima's construction \cite{bell-1986}, see
\cite{darn-junk-2008}. We provide here a
proof using only the most basic properties of Kripke models, and the
finite model property. 
\smallskip

Given a language ${\cal L}$ and a set $\Var$ of variables, an ${\cal
L}$\--term whose variables belong to $\Var$ is called an ${\cal
L}(\Var)$\--term. Remember that Heyting algebras are the algebraic
models of IPC, the intuitionistic propositional calculus. So
$\lha(\Var)$\--terms are nothing but formulas of IPC with
propositional variables in $\Var$, the function symbols of $\lha$
being interpreted as logical connectives in the obvious way, and the
constant symbols $\ZERO$, $\UN$ as $\bot$, $\top$ respectively. 

A {\df Kripke model} is a map $u:P \to \cP(\Var)$ where $\Var$ is a set of
variables, $P$ is an ordered set, and $u$ obeys the following
monotonicity condition\footnote{In the literature the order on $P$ is often reversed. 
We follow here the convention of \cite{ghil-1999} which suits perfectly
well to our purpose.}:
\begin{displaymath}
  q \leq p\enskip \Longrightarrow u(q) \supseteq u(p) 
\end{displaymath}
The Kripke model $u:P \to \cP(\Var)$ is {\df finite} if $P$ is a 
finite set. An {\df isomorphism} with another Kripke model 
$u':P' \to \cP(\Var')$ is an order preserving bijection $\sigma:P \to P'$ such
that $u = u' \circ \sigma$. The notion of an $\lha(\Var)$\--term (or IPC formula) 
$t$ being {\df true at a point $p$ in $u$}, which is denoted $u \Vdash_p t$, 
is defined by induction on $t$:
\begin{displaymath}
  \begin{array}{lcl}
    \lefteqn{u \Vdash_p \top\hbox{ \ and \ }u \nVdash_p \bot,}\\
    u \Vdash_p x       & \iff & x \in u(p),\hbox{ for }x \in \Var, \\
    u \Vdash t_1 \meet t_2 & \iff & u \Vdash t_1\hbox{ and }u \Vdash t_2, \\
    u \Vdash t_1 \join t_2 & \iff & u \Vdash t_1\hbox{ or  }u \Vdash t_2, \\
    u \Vdash t_1 \to t_2 & \iff & \forall q \leq p\ \big(u \Vdash t_1 \Rightarrow u \Vdash t_2\big). 
    \end{array}
\end{displaymath} 
We denote by $\Th(p,u)$ the {\df theory of $p$ in $u$}, that is the set of
$\lha(\Var)$\--terms true at $p$ in $u$. 
If $t$ is true at every point in $u$ we say that $t$ is {\df true in $u$}
and note it $u \Vdash t$. The set of $\lha(\Var)$\--terms true in $u$ is
denoted $\Th(u)$.
Here is the fundamental theorem on Kripke models and IPC (see for
example \cite{popk-1994}):
\begin{theorem}\label{th:fmp-IPC}
  Let $t$ be an $\lha$\--term and $\Var$ be the (finite) set of its
  variables. Then the following are equivalent: 
  \begin{enumerate}
    \item
      $t$ is a theorem of IPC.
    \item 
      $t$ is true in every Kripke model $u:P \to \cP(\Var)$.
    \item 
      $t$ is true in every finite Kripke model $u:P \to \cP(\Var)$.
  \end{enumerate}
\end{theorem} 
The classical duality between finite Kripke models and finite Heyting
algebras (see for example chapter~1 of \cite{fitt-1969}) provides an
algebraic translation of the finite model property. We need to make a
couple of precise observations on this duality, so let us recall it
now in detail. 

Given a Kripke model $u:P \to \cP(\Var)$ and an $\lha(\Var)$\--term $t$ 
we define $u[t]=\{ p \in P \tq u \Vdash_p t \}$. 
The monotonic assumption on $u$ implies by an immediate induction that
$u[t]$ is a decreasing subset of $P$. The family $\cO(P)$ of decreasing
subsets of $P$ is easily seen to be a topology on $P$, hence a Heyting
algebra. Define:
\begin{displaymath}
  G_u=\{ u[x] \tq x \in \Var \}\quad\hbox{and}\quad
  L_u=\{ u[t] \tq t \in \lha(\Var) \}
\end{displaymath}
One can show that $L_u$ is an $\lha$\--substructure of $\cO(P)$ hence 
a Heyting algebra again. For any $\lha$\--term $t(x_1,\dots,x_n)$ and any
elements $u[t_1],\dots,u[t_n]$ in $L_u$:
\begin{displaymath}
  t(u[t_1],\dots,u[t_n]) =  u[t(t_1,\dots,t_n)].
\end{displaymath}
In particular $u[t]=t(u[x_1,\dots,x_n])$ hence $G_u$ is a set of generators 
of $L_u$. Moreover $u \Vdash t(x_1,\dots,x_n)$ if and only if
$t(u[x_1],\dots,u[x_n]) = \UN_{L_u}$. 

Conversely, given a Heyting algebra $L$ with a set of generator $G$ we
can construct a Kripke model as follows. Let $P_L$ be the set of all
prime ideals of $L$, ordered by
inclusion\footnote{Since $\ii \in P_L$ iff $\ii^* \in \Spec L^*$, $P_L$ as an
ordered set is nothing but the prime filter spectrum (ordered by
inclusion) of the co-Heyting algebra $L^*$.}. 
Let $\Var_G$ be any set of variables indexed by $G$. 
For every prime ideal $\ii \in P_L$ define:
\begin{displaymath}
  u_{L,G}(\ii)=\{ x_g \in \Var_G \tq g \notin \ii \} 
\end{displaymath}
Then $u_{L,G}:P_L \to \cP(\Var_G)$ is a Kripke model. Moreover for every
$\lha(\Var_G)$\--term $t=t(x_{g_1},\dots,x_{g_n})$ and every prime ideal 
$\ii \in P_L$:
\begin{displaymath}
   u \Vdash_\ii t(x_{g_1},\dots,x_{g_n})\enskip\iff\enskip t(g_1,\dots,g_n) \notin \ii
\end{displaymath}
In particular $t$ is true in $u$ if and only if $t(g_1,\dots,g_n)=\UN_L$.

Obviously a Kripke model $u$ is finite if and only if $L_u$ is finite,
and a Heyting algebra $L$ is finite if and only if it has finitely
many prime ideals, that is if $P_L$ is finite. So the contraposition
of theorem~\ref{th:fmp-IPC} translates algebraically as follows:

\begin{fact}\label{fa:fmp-HA}
  Let $t$ be an $\lha$\--term. The formula $\exists x,\;t(x) \neq \UN$ has a
  model (a Heyting algebra $L$ in which $t(a) \neq \UN$ for some tuple
  $a$ in $L$) if and only if it has a finite model.
\end{fact}

But there is something more. Observe that for any $\ii \in P_{L,G}$: 
\begin{displaymath}
  \ii = \{ t(g_1,\dots,g_n) \tq t \in \Th(\ii,u_{L,G}) \}
\end{displaymath}
So any two points in $P_L$ having the same theory in $u_{L,G}$ are equal. A
Kripke model having this property will be called {\df reduced}.

Define the {\df length of a Kripke model} $u:P \to \cP(\Var)$ as the maximal
length\footnote{More exactly the length of $u:P \to \cP(\Var)$, or simply the
length of $P$, is the smallest ordinal $\alpha$ such that every element of
$P$ has foundation rank $\leq \alpha$, if such an ordinal exists, and $+\infty$
otherwise.}
of a chain of elements of $P$. Fix a finite set of $n$ variables 
$\Var$ and a positive integer $d$. 
In a Kripke model $u:P \to \cP(\Var)$ of length 0 the theory at any point $p$
is determined by $u(p)$. So if $u$ is reduced it can have at most $2^n$ points. 
Consequently there exists only finitely many non isomorphic reduced
Kripke models of length 0.

Assume that for some positive integers $d$, $\nu$  we have proved that
there exists at most $\nu$ non isomorphic reduced Kripke models of
length at most $d$. Consider a reduced Kripke model $u:P \to \cP(\Var)$ of
length at most $d+1$. For every point $p$ of rank $d+1$ the restriction
$u_{|p\Downarrow}$ of $u$ to $p\Downarrow$ is a reduced Kripke model of length at most
$d$.  If $q$ is another element of rank $d+1$ such that $u_{|q\Downarrow}$ and
$u_{|p\Downarrow}$ are isomorphic then $u(p) \neq u(q)$, otherwise a
straightforward induction would show that $\Th(p,u)=\Th(q,u)$.  So $u$
has at most $2^n \nu$ points of rank $d+1$. Consequently there exists
only finitely many non isomorphic reduced Kripke models of length at
most $d+1$.

Let us say that two $\lha$\--terms $t_1,t_2$ with variables in some
finite set $\Var$ are {\df $d$\--equivalent} if they are true in
exactly the same reduced Kripke model $u:P \to \cP(\Var)$ of length at
most $d$. By the above induction there exists a finite number $\mu$ of
non isomorphic such models, hence at most $2^\mu$  different classes of
$d$\--equivalence. Let us stress this:

\begin{fact}\label{fa:d-equiv-fini}
  For every positive integers $n,d$ there exists finitely many 
  $d$\--equivalence classes of $\lha$\--terms with $n$ variables.
\end{fact}

We can return now to co-Heyting algebras. Let us say that a variety
$\cV$ (in the sense of universal algebra) of co-Heyting algebras has
the {\df finite model property} iff for every $\ltc$\--term $t(x)$, if
there exists an algebra $L$ in $\cV$ such that $\exists x,\, t(x)\neq \ZERO$
holds in $L$ then there exists a finite algebra in $\cV$ having this
property. So fact~\ref{fa:fmp-HA} asserts that the variety of all
co-Heyting algebras has the finite model property.

\begin{proposition}\label{pr:fg-free-TC-codimetric}
  Let $\cV$ be a variety of co-Heyting algebras. The following are
  equivalent:
  \begin{enumerate}
    \item\label{it:var-fmp}%
      $\cV$ has the finite model property.
    \item\label{it:var-free-res-fi}%
      Every algebra free in $\cV$ is residually finite\footnote{A
      co-Heyting algebra $L$ is residually finite if for every non zero
      element $a$ there is an ideal $I$ not containing $a$ such that
      $L/I$ is finite.}.
    \item\label{it:var-free-haus}%
      Every algebra free in $\cV$ is Hausdorff.
    \item\label{it:var-precomp}%
      Every algebra finitely presented in $\cV$ is precompact
      Hausdorff.
  \end{enumerate}
\end{proposition}

\begin{proof}
(\ref{it:var-fmp})$\Rightarrow$(\ref{it:var-free-res-fi})\enskip
Let $\cF$ be an algebra free in $\cV$. Every non zero element of $\cF$
can be written as $t(X)$ for some $\ltc$\--term $t(x)$ and some finite subset $X$ of
the free generators of $\cF$. Since $\cV$ has the finite model property 
there exists a finite algebra $L'$ in
$\cV$ such that $t(a') \neq \ZERO$ for some tuple $a'$ of elements of
$L'$. Let $\varphi:\cF \to L'$ be the unique $\ltc$\--morphism which maps $X$
onto $a'$ and the other generators of $\cF$ to $\ZERO$. Then $I=\Ker
\varphi$ is an ideal of $\cF$ not containing $t(X)$ such that $\cF/I$ is
finite. 
\smallskip

(\ref{it:var-free-res-fi})$\Rightarrow$(\ref{it:var-free-haus})\enskip
Let $\cF$ be an algebra free in $\cV$ and $t(X)$ a non zero element of
$\cF$. The assumption (\ref{it:var-free-res-fi}) gives an ideal $I$ of
$\cF$ not containing $t(X)$ such that $\cF/I$ is finite. Then $\cF/I$ has 
finite dimension, say $d$. 
By corollary~\ref{co:TC-morph-codim}(\ref{it:phi-de-dL}) it follows that $(d+1)\cF$ is
contained in $\Ker \varphi$. So $t(X) \notin (d+1)L$ that is $\codim t(X) \leq d$ is
finite as required.
\smallskip

(\ref{it:var-free-haus})$\Rightarrow$(\ref{it:var-precomp})\enskip
By proposition~\ref{pr:quotient-codimetric} it is sufficient to show
that every free Heyting algebra $\cF$ with a finite set of generators $X$
is precompact. 
Let $t_1(X),t_2(X)$ be any two elements of $\cF$ having different images
in $\cF/d\cF$. Let $Z$ be the image of $X$ in
$\cF/d\cF$ and $Z^*=\{z^*\}_{z\in Z}$ its image in the dual $(\cF/d\cF)^*$. 
By assumption $t_1(Z)\neq t_2(Z)$ {\it ie.\,} 
$t_1(Z) \bigtriangleup t_2(Z) \neq \ZERO$. Dualizing: 
\begin{displaymath}
  (\cF/d\cF)^* \models t^*_1(Z^*) \leftrightarrow t^*_2(Z^*) \neq \UN
\end{displaymath}
where $t^*_i$ is the $\lha$\--term obtained from $t_i$ by dualisation.
Since $\cF/d\cF$ has dimension at most $d$ any chain of prime filters of
$\cF/d\cF$ has length at most $d$. But the prime filters of $\cF/d\cF$ are 
exactly the complements of the prime ideals of its dual $(\cF/d\cF)^*$. 
So the Kripke model $u_{(\cF/d\cF)^*,Z^*}:P_{(\cF/d\cF)^*} \to Z^*$ is a reduced 
Kripke model of height at most $d$ in which $t^*_1 \leftrightarrow t^*_2$ is not true. 

This proves that if $t_1(X),t_2(X)$ have different images
in $\cF/d\cF$ then $t_1^*$, $t_2^*$ are not $d$\--equivalent. By
fact~\ref{fa:d-equiv-fini} there is only a finite number of
$d$\--equivalence classes of $\lha$\--terms with variables in the finite set
$Z^*$ hence $\cF/d\cF$ is finite. 
\smallskip

(\ref{it:var-precomp})$\Rightarrow$(\ref{it:var-fmp})\enskip
Let $t$ be an $\ltc$\--term with $n$
variables such that the formula $\exists x,\, t(x) \neq \ZERO$ holds in some 
algebra $L$ in $\cV$. Let $\cF$ be a free algebra in $\cV$ having an
$n$\--tuple $X$ of generators. The assumption on $t$ implies
that $t(X)\neq\ZERO$. Since $\cF$ is Hausdorff by (\ref{it:var-precomp}), there is a positive
integer $d$ such that $t(X) \notin d\cF$ hence the formula 
$\exists x, t(x)\neq\ZERO$ holds in $\cF/d\cF$ as well, which is finite by
(\ref{it:var-precomp}). 
\end{proof}

\begin{corollary}\label{co:fg-precomp}%
  Every finitely generated co-Heyting algebra is precompact. Every
  finitely presented co-Heyting algebra is precompact Hausdorff.
\end{corollary}

\begin{proof}
If $I$ is any ideal of a $\cV$\--algebra $L$ and $L'=L/I$ then
$L'/dL'$ is also the quotient of $L/dL$ by $\pi_d(I)$. So the
homomorphic image of any precompact co-Heyting algebra is
precompact. Since the variety of all co-Heyting algebras has the
finite model property, the result then follows immediately from
proposition~\ref{pr:fg-free-TC-codimetric}.
\end{proof}

Note that the quotient of a free co-Heyting algebra by any closed ideal 
is Hausdorff by proposition~\ref{pr:quotient-codimetric}, hence a finitely
generated co-Heyting algebra can be Hausdorff without being finitely
presented. 

\begin{example}\label{ex:fg-hausdorff-non-fp}
  Let $\cF_n$ be the free co-Heyting algebra with $n$ generators with
  $n \geq 2$ so that $\cF_n \neq \widehat{\cF_n}$ (the Hausdorff completion
  of $\cF_n$, see section~\ref{se:pre-compact} or the comments after
  fact~3.6 in \cite{darn-junk-2008}). Choose any
  $\widehat{a}$ in $\widehat{\cF_n} \setminus \cF_n$. Then $I = \widehat{a}\downarrow \cap
  \cF_n$ is a closed ideal of $\cF_n$ which is not principal hence
  $\cF_n/I$ is finitely generated and Hausdorff but not finitely
  presented.
\end{example}

\begin{example}\label{ex:fg-non-hausdorff}
  Let $\cF_n$ be as above. For every $n \geq 2$ there are many elements
  in $\cF_n$ which can not be written as the join of finitely many
  join irreducible elements, such as the meet of any two join
  irreducible elements (see remark~4.14 in \cite{darn-junk-2008}). 
  Given any such element $a$, the ideal $I$ generated by the join
  irreducible elements smaller than $a$ is not closed because $a \notin I$
  but $a$ belongs to the topological closure of $I$ (here we use that
  $a = \jjoin \jsupp{}(a)$, see proposition~\ref{pr:precomp-join-irr}). So
  $\cF_n/I$ is finitely generated but not Hausdorff by
  proposition~\ref{pr:quotient-codimetric}.
\end{example}

\section{Precompact Hausdorff co-Heyting algebras}
\label{se:pre-compact}

We have seen that every finitely presented co-Heyting algebra is
precompact Hausdorff, but the latter form a much larger class. It is
then remarkable that most of the very nice algebraic properties of
finitely presented free Heyting algebras obtained in
\cite{darn-junk-2008} from \cite{bell-1986} actually generalise, after
dualisation, to precompact Hausdorff co-Heyting algebras.

\phantomsection
\addcontentsline{toc}{subsection}%
   {Precompactness and profinite completion}
\subsection*{Precompactness and profinite completion}

Let $L$ be a co-Heyting algebra, $d$ a positive integer and
$L'=L/(d+1)L$. Corollary~\ref{co:TC-morph-codim}(\ref{it:phi-de-dL})
asserts that $\pi_{dL}^{-1}(dL')=dL$ hence (see
remark~\ref{re:identification-quotient}) $L'/dL'$ identifies with
$L/dL$ and $\pi_{dL'}$ with a surjective map that we denote:
\begin{displaymath}
  \pi_{d,d+1}:L/(d+1)L \to L/dL
\end{displaymath}
Similarly, in order to make the reading easier, we let $\pi_d$ denote
$\pi_{dL}$ for every positive integer $d$. So $\{\pi_{d,d+1}:L/(d+1)L \to
L/dL\}_{d<\omega}$ is a projective system and the following diagram is
commutative:
\begin{displaymath}
   \UseTips
 \newdir{ >}{!/-5pt/\dir{>}}
\xymatrix{
   & & & & L \ar[dd]_{\pi_0} \ar[lldd]_{\pi_d} \ar[llldd]_{\pi_{d+1}} \\
   & & & & \\
   \cdots \ar[r] & L/(d+1)L \ar[r]_{\ \ \ \pi_{d,d+1}} & L/dL \ar[r] & \cdots \ar[r] & L/0L } %
\end{displaymath}
We denote by $\widehat{L}$ its projective limit. Note that the
canonical map from $L$ to $\widehat{L}$ is an embedding if and only if
$L$ is Hausdorff. The codimetric topology on each $L/dL$ is the
discrete topology. We equip $\widehat{L}$ with the corresponding
projective topology. As a projective limit of Hausdorff topologies,
this topology on $\widehat{L}$ is Hausdorff and the image of $L$ in
$\widehat{L}$ is dense in $\widehat{L}$. It will be shown in
section~\ref{se:completion} that
$\widehat{L}$ is nothing but the Hausdorff completion of $L$. 
However, when $L$ is precompact Hausdorff, the proof that we provide
below is much simpler.
\smallskip

Assume that $L$ is precompact Hausdorff. Then the projective topology
on $\widehat{L}$ is profinite hence compact Hausdorff. We refer the
reader to any book of topology for this and the following classical
results on projective limits of topological spaces. We identify $L$ 
with its image in $\widehat{L}$ {\it via} the diagonal embedding. We
denote by $\overline{\pi}_d$ (resp. $\widehat{\pi}_d$) the canonical
projection of $\widehat{L}$ onto $L/dL$ (resp.
$\widehat{L}/d\widehat{L}$). Obviously $\pi_d$ is the restriction of
$\overline{\pi}_d$ to $L$ and the kernel of $\overline{\pi}_d$ is the
topological closure $\overline{dL}$ of $dL$ in $\widehat{L}$. 

\begin{displaymath}
   \UseTips
 \newdir{ >}{!/-5pt/\dir{>}}
\xymatrix{
   L  \ar[d]_{\pi_d}\ar[rr]  & &
   \ar[lld]_{\overline{\pi}_d}\ar[d]^{\widehat{\pi}_d} \widehat{L} \\
   L/dL                    & & \widehat{L}/d\widehat{L} }
\end{displaymath}

\begin{theorem}\label{th:precomp-dL}%
  Let $L$ be a Hausdorff precompact co-Heyting algebra. Then for
  every positive integer $d$:
  \begin{enumerate}
    \item
      $dL$ and $d\widehat{L}$ are principal\,\footnote{So $dL$ is
      principal for every finitely presented
      co-Heyting algebra $L$, by corollary~\ref{co:fg-precomp}. This
      is actually true for finitely generated co-Heyting algebras, as we
      will show in section~\ref{se:appendix}.} and
      $\varepsilon_d(L)=\varepsilon_d(\widehat{L})$;
    \item 
      $\overline{dL}=d\widehat{L}$ and $\widehat{L}/d\widehat{L}$
      identifies with $L/dL$.
  \end{enumerate}
  As a consequence the projective topology on $\widehat{L}$ coincides
  with its codimetric topology and $(\widehat{L},\dist{\widehat{L}})$
  is the completion of the metric space $(L,\dist{L})$.
\end{theorem}

We first need a lemma.
Recall that $\widehat{L}$ can be represented as:
\begin{displaymath}
  \widehat{L} = \bigl\{(x_k)_{k<\omega} \tq \forall k,\ x_k \in L/kL
                       \mbox{ and }\pi_{k,k+1}(x_{k+1})=x_k\bigr\}
\end{displaymath}
Note that if $L$ is precompact then for every positive integer $d$ and
every $k$, $d(L/kL)$ is obviously principal because $L/kL$ is finite. 
Moreover $\big(\varepsilon_d(L/kL)\big)_{k<\omega}$ belongs to $\widehat{L}$ by
corollary~\ref{co:TC-morph-codim}(\ref{it:phi-de-epsilon}), using the
above representation of $\widehat{L}$. Let us denote by $\widehat{\varepsilon}_d$ 
this element of $\widehat{L}$. Note that
$\widehat{\varepsilon}_{d+1} \ll \widehat{\varepsilon}_d$ because $\varepsilon_{d+1}(L/kL) \ll \varepsilon_d(L/kL)$
for every $k$ by corollary~\ref{co:epsilon-ll}. 
A basis of neighborhood of
any element $x \in \widehat{L}$ is given as $d$ ranges over the positive
integers, by\footnote{We simply use here that in each $L/kL$, a basis
of neighborhood of $x_k$ with respect to the discrete/codimetric topology 
is given by $\{U(x_k,d)\}_{d<\omega}$ (see (\ref{eq:Uxd}) in
section~\ref{se:pseudometric}).}:
\begin{displaymath}
  B(x,d)=\{y \in \widehat{L} \tq x \bigtriangleup y \leq \widehat{\varepsilon}_d\}
\end{displaymath}
In particular $\{\widehat{\varepsilon}_d\downarrow\}_{d<\omega}$ is a basis of neighborhood of
$\ZERO$ in $\widehat{L}$.

\begin{lemma}\label{le:precomp-isole}
  Let $L$ be a Hausdorff precompact co-Heyting algebra. Then an
  element $x \in \widehat{L}$ is isolated (with respect to the
  projective topology) if and only if $\widehat{\varepsilon}_d
  \leq x$ for some $d$. In this case $x \in L$. 
  In particular $\widehat{\varepsilon}_d \in L$ for every $d$.
\end{lemma}

\begin{proof}
Let $x=(x_k)_{k<\omega}$ be any element of $\widehat{L}$. 
If $x$ is isolated in $\widehat{L}$ then for some
integer $d$ we have $B(x,d)=\{x\}$. On the other hand 
$x \bigtriangleup (x \join \widehat{\varepsilon}_d) =\widehat{\varepsilon}_d - x \leq \widehat{\varepsilon}_d$, 
that is $x \join \widehat{\varepsilon}_d \in B(x,d)$ so we are done. 

Conversely assume that $x = x \join \widehat{\varepsilon}_d$ for some $d$. Then 
$\widehat{\varepsilon}_{d+1} \ll \widehat{\varepsilon}_d \leq x$ hence $x-\widehat{\varepsilon}_{d+1}=x$. For every 
$y \in B(x,d+1)$ we get:
\begin{displaymath}
  x = x - \widehat{\varepsilon}_{d+1} 
    = \big[(x \meet y) \join (x - y)\big]-\widehat{\varepsilon}_{d+1}
    = (x \meet y) - \widehat{\varepsilon}_{d+1} 
    \leq y
\end{displaymath}
And:
\begin{displaymath}
    y = (x \meet y) \join (y -x) \leq x \join  \widehat{\varepsilon}_{d+1} = x
\end{displaymath}
This proves that $B(x,d+1)=\{x\}$.

The last assertion follows because $L$ is dense in $\widehat L$ for the
projective topology, and an isolated point in a topological space
obviously belongs to every dense subspace. 
\end{proof}

We can now achieve the proof of theorem~\ref{th:precomp-dL}.

\begin{proof}
For every positive integer $d$ we have $dL \subseteq d\widehat{L}$ by
corollary~\ref{co:TC-morph-codim}(\ref{it:morph-lipshitz}) because the 
inclusion is an $\ltc$\--morphism. Moreover
$d\widehat{L} \subseteq \overline{dL}$ by
corollary~\ref{co:TC-morph-codim}(\ref{it:phi-de-dL}) because 
$\dim L/dL <d$ and $\overline{dL}=\Ker \overline{\pi}_d$.

By construction the ideal generated in $\widehat L$ by $\widehat{\varepsilon}_d$
is precisely $\overline{dL}$. By lemma~\ref{le:precomp-isole}
$\widehat{\varepsilon}_d$ actually belongs to $L$. Moreover it belongs to $dL$
because:
\begin{displaymath}
  \widehat{\varepsilon}_d \ll \cdots \ll \widehat{\varepsilon}_0 = \UN 
\end{displaymath}
Since $dL \subseteq d\widehat{L} \subseteq \overline{dL}$ it immediately follows that 
$d\widehat{L} = \overline{dL}$ hence $\varepsilon_d(\widehat{L})=\widehat{\varepsilon}_d$. 
Moreover $dL = \overline{dL} \cap L$ hence $dL = \widehat{L} \cap L$. We
conclude that $\varepsilon_d(L)=\widehat{\varepsilon}_d$.

The identification of $\widehat{L}/\widehat{dL}$
with $L/dL$ follows  since
$\widehat{\pi}_d$ and $\overline{\pi}_d$ have the same kernel. 

We have proved that $d\widehat{L} = \widehat{\varepsilon}_d\downarrow$ hence
$B(x,d)= U(x,d)$ (see (\ref{eq:Uxd}) in section~\ref{se:pseudometric}) 
for every positive integer $d$ and every $x \in \widehat{L}$. As a
consequence the projective topology on $\widehat{L}$ coincide with its
codimetric topology. Since $\widehat{L}$ is compact, it is complete,
and since $L$ is dense in $\widehat{L}$ the last statement follows.
\end{proof}

\phantomsection
\addcontentsline{toc}{subsection}%
   {Join irreducible elements}
\subsection*{Join irreducible elements}

We denote as follows the sets of {\df join irreducible}, {\df
completely join irreducible}, {\df meet irreducible} and {\df
completely meet irreducible} elements respectively: 
\begin{eqnarray*}
  \jirr(L)  &=& \{x \in L\setminus\{\ZERO\}\tq \forall a,b \in L,\ x \leq a \join b
                     \Rightarrow x \leq a \hbox{ or } x \leq b\} \\
  \cjirr(L) &=& \{x \in L\setminus\{\ZERO\}\tq \forall A \subseteq L,\ x \leq \jjoin A \Rightarrow \exists a \in A,\ x \leq a\} \\
  \mirr(L)  &=& \{x \in L\setminus\{\UN\}\tq \forall a,b \in L,\ a \meet b \leq x
                     \Rightarrow a \leq x \hbox{ or } b \leq x\} \\
  \cmirr(L) &=& \{x \in L\setminus\{\UN\}\tq \forall A \subseteq L,\ \mmeet A \leq x \Rightarrow \exists a \in A,\ a \leq x\}
\end{eqnarray*}

\begin{remark}\label{re:ll-join-irr}
  If $x$ is join irreducible and $x \nleq y$ then $x-y=x$. Indeed $x \meet y
  <x$ and $x=(x-y)\join(x \meet y)$, then use the join irreducibility of $x$.
  In particular $y \ll x$ whenever $y < x$.
\end{remark}

The following lemma is folklore. 

\begin{lemma}\label{le:quotient-fini}
  Let $\varepsilon$ be any element of a co-Heyting algebra $L$, let $L'$ be
  the quotient of $L$ by the ideal $\varepsilon\downarrow$ and let $\pi:L \to L'$ be the canonical
  projection. 
  \begin{enumerate}
    \item
      $\forall a \in L$, $a - \varepsilon = \min \pi^{-1}(\{\pi(a)\})$ and 
      $a \join \varepsilon = \max \pi^{-1}(\{\pi(a)\})$.

       So the restrictions of $\pi$ to $\{a-\varepsilon\}_{a\in L}$ and 
       $\{a \join \varepsilon\}_{a \in L}$ are one-to-one.
    \item  
      If in addition $L'$ is finite then every prime filter (resp.
      ideal) of $L$ disjoint from $\varepsilon\downarrow$ (resp. containing $\varepsilon\downarrow$) is
      generated by a completely join (resp. meet) irreducible element.
      So $\pi$ induces a one-to-one order preserving correspondence
      between the following sets:
      \begin{eqnarray*}
        \mirr(L') &\longleftrightarrow& \{x \in \cmirr(L) \tq \varepsilon \leq x\}  \\
        \jirr(L') &\longleftrightarrow& \{x \in \cjirr(L) \tq x \nleq \varepsilon\}  
      \end{eqnarray*}
  \end{enumerate}
\end{lemma}

\begin{proof}
For every $x,y$ in $L$, $\varphi(y)\leq\varphi(x) \iff y-x \leq \varepsilon$. The first point then
follows from straightforward calculations:
\begin{displaymath}
  y-\varepsilon \leq a \iff y \leq x \join \varepsilon \iff y-x \leq \varepsilon
\end{displaymath}

Now assume that $L'$ is finite. Then every prime ideal of $L'$ is
generated by a completely meet irreducible element. As a surjective
$\ltc$\--morphism,
$\pi$ induces a one-to-one order preserving correspondence between the prime ideals of
$L$ containing $\varepsilon\downarrow$ (its kernel) and the prime ideals of $L'$ (its
image) which preserves inclusions. So it is
sufficient to prove that, given an element $x'$ of $L'$ having a
unique successor $x'^ +$, the ideal $\varphi^{-1}(x'\downarrow)$ is generated by an
element having a unique successor. In order to do this let $x$ (resp.
$a$) be any element of $L$ such that $\varphi(x)=x'$ (resp. $\varphi(a)=x'^+$). For
every $b \in L$ we have: 
\begin{displaymath}
  \varphi(b) \in x'\downarrow \iff \varphi(b) \leq \varphi (x) \iff b \leq x \join \varepsilon 
\end{displaymath}
So $x \join \varepsilon$ is the generator of $\varphi^{-1}(x'\downarrow)$. Moreover:
\begin{eqnarray*}
  x \join \varepsilon < b \join \varepsilon &\iff& x'= \varphi(x) < \varphi(b) \\
                &\iff& x'^+= \varphi(a) \leq \varphi(b) \\
                &\iff& a \join \varepsilon \leq b \join \varepsilon
\end{eqnarray*}
So $a \join \varepsilon$ is the unique successor of $x \join \varepsilon$ in $L$. 

The case of join irreducible elements is similar: $\pi$ induces a
one-to-one order preserving correspondence between the prime filters
disjoint from $\varepsilon\downarrow$ and the prime filters of $L'$. Given an element
$x'\in L'\setminus\{\ZERO\}$ having a unique predecessor $x'^-$, the inverse image by
$\pi$ of $x'\uparrow$ is generated by an element $x$ having a unique
predecessor. We take any two elements $x,a \in L$ such that $\pi(x)=x'$
and $\pi(a)=x'^-$. The reader may easily check that $x - \varepsilon$ is a
generator of $\pi^{-1}(x'\uparrow)$ and $a-\varepsilon$ is its unique predecessor. 
\end{proof}

\begin{remark}\label{re:precomp-corresp-pi-d}%
  If $L$ is a precompact Hausdorff co-Heyting algebra 
  and $d$ a positive integer then $\Ker \pi_d = \varepsilon_d(L)\downarrow$ by 
  theorem~\ref{th:precomp-dL}. Then lemma~\ref{le:quotient-fini} applied to
  $\varepsilon_d(L)$ tells us that every join (resp.\ meet) irreducible
  element of $L \setminus dL$ (resp. of $\varepsilon_d(L)\uparrow$) is completely join 
  (resp.\ meet) irreducible, and that $\pi_d$ induces a one-to-one 
  correspondence between the following sets:
      \begin{eqnarray*}
        \jirr(L/dL) &\longleftrightarrow& \cjirr(L) \setminus dL \\
        \mirr(L/dL) &\longleftrightarrow& \cmirr(L) \cap \varepsilon_d(L)\uparrow
      \end{eqnarray*}
  These sets are finite, in particular there are finitely many
  completely join irreducible elements in $L$ of any given finite
  codimension.
\end{remark}
Given an element $a \in L$ the maximal elements of $\jirr(L) \cap a\downarrow$, if
they exist, are called the {\df join irreducible components} of $a$ in
$L$. The set of join irreducible components of $a$ is denoted
$\jsupp{L}(a)$. As usually the index $L$ is often omitted.
The {\df meet irreducible components} of $a$ in $L$ and the set
$\msupp{L}(a)$ are defined dually. 

\begin{proposition}\label{pr:precomp-join-irr}%
  Let $L$ be a precompact Hausdorff co-Heyting algebra.
  \begin{enumerate}
    \item 
      $L$ and $\widehat{L}$ have the same completely join irreducible
      elements.
    \item 
      Every join irreducible element of $L$ is completely join
      irreducible.
    \item 
      For every $x \in \cjirr(L)$, the cofoundation rank of $x$ in
      $\cjirr(L)$ is finite. It is the codimension of $x$.
    \item 
      $\cjirr(L)$ satisfies the ascending chain condition.
    \item 
      For every $a \in L$, $a = \jjoin \jsupp{} a$.
  \end{enumerate}
\end{proposition}

\begin{proof}
Since $L/dL=\widehat{L}/d\widehat{L}$ for every $d$ and 
$\bigcap_{d<\omega}dL=\{\ZERO\}$, the two first points follow immediately from
lemma~\ref{le:quotient-fini} applied to $\varepsilon_d(L)$ (see
remark~\ref{re:precomp-corresp-pi-d}). For the third point, note
simply that it is true in every finite lattice, because every prime
filter is generated by a completely join irreducible element,  and apply
lemma~\ref{le:quotient-fini} with $\varepsilon=\varepsilon_d(L)$ for any $d$ such that 
$x \nleq \varepsilon_d(L)$. The ascending chain condition follows: every element in
$\cjirr(L)$ has finite corank because it has finite codimension.

For the last point, fix an element $a \in L\setminus\{\ZERO\}$. For every positive
integer $d$, let:
\begin{displaymath}
  a_d= \jjoin\{x \in \jirr(L) \tq x \leq a \hbox{ and }x \nleq \varepsilon_d(L)\}
\end{displaymath}
By lemma~\ref{le:quotient-fini}, $a_d = a - \varepsilon_d(L)$ hence by continuity 
of $x \mapsto a-x$ the sequence $(a_d)_{d<\omega}$ is convergent to $a$. So $a$
is the complete join of all the join irreducible elements of $L \cap  a\downarrow$.
These elements are completely join irreducible, hence by the ascending
chain condition each of them is smaller than a maximal one, which
proves the last point.
\end{proof}

\phantomsection
\addcontentsline{toc}{subsection}%
   {Meet irreducible elements}
\subsection*{Meet irreducible elements}

The case of meet irreducible elements in a precompact Hausdorff
co-Heyting algebra is slightly more complicated. For example they are not
always completely irreducible, contrary to the join irreducible elements 
(see proposition~\ref{pr:meet-non-complete} below). 
\smallskip

In finite distributive lattices there is a correspondence between
(completely) join and meet irreducible elements which is defined as
follows. For every $x \in L$ let:
\begin{displaymath}
  x^\join = \mmeet \{y \in L \tq y \nleq x\}\qquad
  x^\meet = \jjoin \{y \in L \tq x \nleq y\}
\end{displaymath}
Then $x \in \cmirr(L) \Rightarrow x^\join \in \cjirr(L)$ and symmetrically 
$x \in \cjirr(L) \Rightarrow x^\meet \in \cmirr(L)$. These two operations are easily
seen to define reciprocal, order preserving bijections between
$\cmirr(L)$ and $\cjirr(L)$. 

This correspondence generalizes to join complete and meet complete
lattices which satisfy the infinite distributive laws:
\begin{displaymath}
  x \meet \jjoin_{y \in Y} y = \jjoin_{y \in Y} (x \meet y) \qquad 
  x \join \mmeet_{y \in Y} y = \mmeet_{y \in Y} (x \join y)
\end{displaymath}
In particular it holds for profinite lattices, and we take advantage
of this in the following proposition. 

\begin{proposition}\label{pr:precomp-meet-irr}%
  Let $L$ be a precompact Hausdorff co-Heyting algebra.
  \begin{enumerate}
    \item 
      $L$ and $\widehat{L}$ have the same completely meet irreducible
      elements.
    \item 
      $x \mapsto x^\join$ and $x \mapsto x^\meet$ are well-defined, reciprocal, order
      preserving bijections between $\cjirr(L)$ and $\cmirr(L)$.
    \item 
      For every $x \in \cmirr(L)$, the cofoundation rank of $x$ in
      $\cmirr(L)$ is finite. 
    \item 
      $\cmirr(L)$ satisfies the ascending chain condition.
    \item 
      Every element $a \in L$ is the complete meet of $\cmirr(L)\cap a\uparrow$. 
  \end{enumerate}
\end{proposition}

\begin{proof}
For every element $\hat a$ in $\widehat{L}$ and every positive
integer $d$, lemma~\ref{le:quotient-fini} applied to
$\varepsilon_d(\widehat{L})$ shows that 
$\mirr(\widehat{L}) \cap (\hat a \join \varepsilon_d(\widehat{L}))\uparrow$ is finite,
contained in $\cmirr(\widehat{L})$, and its complete meet is equal to 
$\hat a \join \varepsilon_d(\widehat{L})$. The sequence $\hat a \join \varepsilon_d(\widehat{L})$
is convergent to $\hat a$ hence: 
\begin{displaymath}
  \hat a = \mmeet \{ x \in \mirr(\widehat{L}) \tq \exists d < \omega,\ \varepsilon_d(\widehat{L}) \leq x \}
\end{displaymath}
It follows that if $\hat a$ is completely meet irreducible, it must be
greater than $\varepsilon_d(\widehat{L})$ for some $d$, hence it belongs to $L$ by
lemma~\ref{le:precomp-isole}. Conversely if $a \in L$ is completely meet 
irreducible in $L$ then by the above equality and
lemma~\ref{le:precomp-isole} it must be greater than $\varepsilon_d(L)$ for some
$d$. The filter generated by $\varepsilon_d(L)$ in $\widehat{L}$ is finite by
lemma~\ref{le:quotient-fini} and contained in $L$ by
lemma~\ref{le:precomp-isole} hence $a$ remains completely meet
irreducible in $\widehat{L}$. This proves the first and the last
point.

Since $\cjirr(L)=\cjirr(\widehat{L})$ and
$\cmirr(L)=\cmirr(\widehat{L})$, $L$ inherits from the profinite
lattice $\widehat{L}$ the correspondence between $\cjirr(L)$ and
$\cmirr(L)$. This proves the second point, and the remaining points
then follow from proposition~\ref{pr:precomp-join-irr}.
\end{proof}

Proposition~\ref{pr:precomp-join-irr} shows that the cofoundation rank of any
completely join irreducible inside $\cjirr(L)$ is equal to its
codimension in $L$. There is a symmetric interpretation for the
cofoundation rank in $\cmirr(L)$. 

\begin{proposition}\label{pr:corang-irr}
  Let $L$ be a precompact Hausdorff co-Heyting algebra, and $x \in
  \cmirr(L)$. Let $r$ be its cofoundation rank in $\cjirr(L)$. Then:
  \begin{displaymath}
    \dim_{L^*} x^* = r = \codim_L x^\join 
  \end{displaymath}
\end{proposition}

\begin{proof}
The cofoundation rank of $x$ in $\cjirr(L)$ is the cofoundation rank of
$x^\join$ in $\cmirr(L)$, because the map $y\mapsto y^\join$ from $\cmirr(L)$ to
$\cjirr(L)$ is one-to-one and order preserving. We have seen in
proposition~\ref{pr:precomp-join-irr} that the
latter is the codimension of $x^\join$ in $L$, so the second equality is
proved.

Note that the prime filters of $L^*$ are exactly the sets $\ii^*$
where $\ii$ is a prime ideal of $L$. Since $x$ belongs to
$\cmirr(L)$, we get that $x^*\in\cjirr(L^*)$, hence the dimension of
$x^*$ in $L^*$ is exactly the height of the prime filter generated by
$x^*$ in $L^*$. Now a prime filter $\ii^*$ of $L^*$ contains $x^*$
if and only if the corresponding prime ideal $\ii$ of $L$ contains
$x$. By proposition~\ref{pr:precomp-meet-irr}, $x$ is greater than
$\varepsilon_d(L)$ for some $d$, hence $L/(x\downarrow)$ is finite. Then by
lemma~\ref{le:quotient-fini} every prime ideal of $L$ containing $x$
is generated by a completely meet irreducible element. Since $x \leq y$
if and only if $x\downarrow\subseteq y\downarrow$, it follows that the height of $(x^*)\uparrow$ in
$\Spec L^*$ is exactly the cofoundation rank of $x$ in $\cmirr(L)$. 
\end{proof}

\begin{remark}\label{re:dim-duale}
  One may wonder what are $\dim x$ for $x\in \cjirr(L)$, and
  $\codim_{L^*} y^*$ for $y \in \cmirr(L)$. They do have a good
  behaviour when $L$ and $L^*$ are finite dimensional. However the
  special case of $\cF_n$, the free co-Heyting algebra with $n$
  generators, shows that although $\cF_n$ is bi\--Heyting, these
  notions do not provide any significant information, contrary to the
  codimension. Indeed one can prove that the foundation
  rank of $x$ in $\cjirr(\cF_n)$ is $+\infty$. The cofoundation rank of
  $y^*$ in $\cjirr(\cF_n^*)$ is also the foundation rank of $y$ in
  $\cmirr(\cF_n)$, which is $+\infty$ as well (see \cite{darn-junk-2008},
  comments after lemma~4.1). It follows that: 
  \begin{displaymath}
    \dim_{\cF_n}x=\codim_{\cF_n^*}y^*=+\infty
  \end{displaymath}
  As a consequence of this and propositions~\ref{pr:precomp-join-irr}
  and \ref{pr:precomp-meet-irr}, $\dim_{\cF_n}a$ and
  $\codim_{\cF_n^*}a^*$ are $+\infty$ for every element 
  $a \in \cF_n\setminus\{\ZERO\}$, and $\dim_{\cF_n^*}a^*$ is finite only if $a$
  is a finite meet of completely meet irreducible elements, or
  equivalently if $a \geq \varepsilon_d(\cF_n)$ for some $d$. 
\end{remark}

In every distributive lattice, if an element $x$ is the complete meet
of a set $Y$ of meet irreducible elements such that $Y$ is downward
filtering\footnote{An ordered set $Y$ is downward filtering if for
every $y_1,y_2 \in Y$ there exists $y\in Y$ smaller than $y_1$ and $y_2$.}
then $x$ itself is meet irreducible. Indeed if $x_1 \meet x_2 \leq x$, 
$x_1 \nleq x$ and $x_2 \nleq x$, let $y_1,y_2 \in Y$ such that $x_1 \nleq y_1$ and 
$x_2 \nleq y_2$. The assumption on $Y$ gives $y \in Y$ smaller than $y_1 \meet y_2$.
Then $x_1 \meet x_2 \leq x \leq y$ hence $x_1 \leq y$ or $x_2 \leq y$ (because $y$ is
meet irreducible) so $x_1 \leq y_1$ or $x_2 \leq y$, a contradiction. 

In particular, if $L$ is a precompact Hausdorff co-Heyting algebra,
then the complete meet in $\widehat L$ of any chain of completely
meet irreducible elements is meet irreducible. By Zorn's lemma it
follows that for every $a \in \widehat L$, every element in
$\cmirr(\widehat{L})\cap a\uparrow$ is greater than a minimal one. So the last
point of proposition~\ref{pr:precomp-meet-irr} leads to:

\begin{corollary}
  Let $L$ be a precompact Hausdorff co-Heyting algebra. For every
  $a\in \widehat{L}$, $a = \mmeet \msupp{} a$.
\end{corollary}

We turn now to a characterisation of the meet irreducible elements of $L$. 

\begin{proposition}\label{pr:meet-non-complete}%
  Let $L$ be a precompact Hausdorff co-Heyting algebra.
  \begin{enumerate}
    \item
      An element $a \in L$ is meet irreducible if and only if $\cmirr(L)
      \cap a\uparrow$ is downward filtering.
    \item 
      A meet irreducible element is completely meet irreducible if and
      only if its cofoundation rank in $L$ (with respect to the strict order
      $<$ of $L$) is finite. 
  \end{enumerate}
  In particular if $L$ is not finite then $\ZERO$ is 
  meet irreducible, but not completely meet irreducible.
\end{proposition}

\begin{proof}
Since $a$ is the complete meet of $\cmirr(L)\cap a\uparrow$, if this set
is downward filtering then $a$ is meet irreducible by the above
general argument. Conversely assume that $a$ is meet irreducible. Let
$y_1,y_2 \in \cmirr(L)\cap a\uparrow$ and $b_i = a \join y_i^\join$. By definition $y_i^\join
\nleq a$ since $a \leq y_i$, hence $y_1^\join \meet y_2^\join \nleq a$ (because $a$ is meet
irreducible). So by proposition~\ref{pr:precomp-meet-irr} there is a
join irreducible component $x$ of $y_1^\join \meet y_2^\join$ which is not smaller
than $a$. By construction $x^\meet \leq y_i$ because $x \leq y_i^\join$ (and
$y_i=y_i^{\join\meet}$). Moreover $a \leq x^\meet $ because $x \nleq a$ (by definition
of $x^\meet$). So $x \in \cmirr(L)\cap a\uparrow$ and the first point is proved. 

Assume now that $a$ meet irreducible. If its cofoundation rank in $L$ is
finite then $\cmirr(L)\cap a\uparrow$ is finite. Since it is downward
filtering it must have a smallest element, hence $a$ is completely
meet irreducible. Conversely if $a$ is completely meet irreducible,
then it is greater than $\varepsilon_d(L)$ for some $d$ (see the proof of the
first point of proposition~\ref{pr:precomp-meet-irr}). But
$\varepsilon_d(L)\uparrow$ is finite by lemma~\ref{le:quotient-fini} (because $L/dL$ is
finite) hence so is $a\uparrow$.
\end{proof}

It was proven in \cite{bell-1986} that in free finitely generated
co-Heyting algebras every join irreducible element is meet irreducible.
This obviously does not hold for finite co-Heyting algebras, hence it
does not generalize to precompact Hausdorff ones.

\phantomsection
\addcontentsline{toc}{subsection}%
   {The smallest dense subalgebra}
\subsection*{The smallest dense subalgebra}

\begin{proposition}\label{pr:precomp-skeleton}
  Let $L$ be a precompact Hausdorff co-Heyting algebra. Then
  $\cjirr(L)$ and $\cmirr(L)$ generate the same $\ltc$\--substructure
  of $L$, which is also the smallest $\ltc$\--substructure dense in
  $L$ (with respect to the codimetric topology). 
\end{proposition}

\begin{proof}
Let $L^\join$ (resp. $L^\meet$) be the $\ltc$\--substructure of $L$  generated
by $\cjirr(L)$ (resp. $\cmirr(L)$). 

Note that if an element $x$ of $L$ is greater than $\varepsilon_d(L)$ for some
$d$ then $\jsupp{}x$ is contained in $\cjirr(L)\setminus(d+1)L$ hence is
finite by remark~\ref{re:precomp-corresp-pi-d}. Moreover:
\begin{displaymath}
  \mirr(L)\cap x\uparrow \subseteq \mirr(L) \cap \varepsilon_d(L)\uparrow = \cmirr(L) \cap \varepsilon_d(L)\uparrow
\end{displaymath}
so $\mirr(L)\cap x\uparrow$ is finite also and contained in $\cjirr(L)$. It follows
that every isolated point of $L$, and in particular every $\varepsilon_d(L)$, 
belongs both to $L^\join$ and $L^\meet$. 

In particular $\cmirr(L) \subset L^\join$ hence  $L^\meet \subseteq L^\join$. Conversely if $x$
is any completely join irreducible element of $L$ and $d > \codim x$
then $x = x - \varepsilon_d(L)$. Because $\varepsilon_d(L)$ and $x - \varepsilon_d(L)$ are isolated
they belong to $L^\meet$, so:
\begin{displaymath}
  x = x - \varepsilon_d(L) = (x \join \varepsilon_d(L)) - \varepsilon_d(L) \in L^\meet
\end{displaymath}
It follows that $L^\join = L^\meet$ is also the $\ltc$\--substructure
generated by the set of isolated points, hence it is contained in
every dense $\ltc$\--substructure of $L$. Conversely every $x \in L$ is
the limit of $(x \join \varepsilon_d(L))_{d<\omega}$ which is a sequence of isolated
points, hence $L^\join$ is dense in $L$.
\end{proof}

\begin{proposition}
  Given a precompact Hausdorff co-Heyting algebra $L$, and $L^\join$ its
  smallest dense subalgebra, the following conditions are equivalent:
  \begin{enumerate}
    \item
      $\widehat{L}=L^\join$.
    \item 
      $\widehat{L}$ is countable or finite.
    \item 
      There is no infinite antichain in $\cjirr(L)$.
    \item 
      There is no infinite antichain in $\cmirr(L)$.
  \end{enumerate}
\end{proposition}

\begin{proof}
The equivalence of the two last conditions follows immediately from the
one-to-one, order preserving correspondence between $\cjirr(L)$ and
$\cmirr(L)$ (see proposition~\ref{pr:precomp-meet-irr}). If there is
an infinite antichain $(x_i)_{i<\omega}$ in $\cjirr(L)$ then for every
subset $I$ of $\NN$, the complete join $x_I$ of $(x_i)_{i \in I}$ belongs
to $\widehat L$ since $\widehat L$ is join complete. These elements are two
by two distinct hence $\widehat L$ is uncountable. Conversely, note
that the join irreducible components of any element form an antichain,
and $\cjirr(L)=\cjirr(\widehat{L})$. So the third condition implies
that $\widehat{L}=L^\join$ which is obviously countable our finite. 
\end{proof}

If $L$ is a precompact Hausdorff co-Heyting algebra such that $L \neq
L^\join$ then obviously $L \not\simeq L^\join$ (because the latter does not contain
a proper dense subalgebra). Because of the density of $L^\join$ in $L$,
both of them satisfy the same identities, an argument that we will
re-use and develop in section~\ref{se:completion}. Does it happen that
$\widehat L \neq L^\join$ but $\widehat L \equiv L^\join$? Our guess is no. But the
analogy with the model theory of the ring $\ZZ_p$ of $p$\--adic numbers
(which is both the completion of $\ZZ$ with respect to the $p$\--adic
ultrametric distance, and the projective limit of all the quotients
$\ZZ/p^d\ZZ$) suggests the following questions.

\begin{question}
  Is the existential closure of $L^\join$ inside $\widehat L$ an
  elementary substructure of $L$~?
\end{question}

\begin{question}
  When $L$ is finitely presented, is $L$ the existential closure of
  $L^\join$ inside $\widehat L$?
\end{question}

In \cite{darn-junk-2008} it was proven that if the free co-Heyting
algebra $\cF_n$ with $n$ generators is elementarily equivalent to
$\widehat\cF_n$ then $\cF_n\preccurlyeq\widehat{\cF_n}$. More
generally, does this hold for every precompact Hausdorff co\--Heyting
algebra?

\section{Hausdorff completion}
\label{se:completion}

Since the Hausdorff completion $L'$ of a co-Heyting
algebra $L$ is the completion of $L/\omega L$ we can assume 
w.l.o.g.\ that $L$ is Hausdorff. We identify $L$ with its image in $L'$ and
consider it as a dense subset of $L'$. By
proposition~\ref{pr:term-lipshitz} the $\ltc$\--functions $\join,\meet,-$ are
continuous on $L \times L$. Their unique continuous extension to $L' \times L'$
defines an $\ltc$\--structure on $L'$. Moreover for any two
$\ltc$\--terms $t_1,t_2$ with $n$ free variables, if the corresponding
functions coincide on $L^n$ then by continuity (and density inside
$L'^n$) they coincide on $L'^n$. So every equation $t_1(x)=t_2(x)$
valid on the whole of $L^n$ remains valid on $L'^n$. Since the class
of all co-Heyting algebras is a variety, it can be axiomatized by
equations. It follows that $L'$ with this $\ltc$\--structure is a
co-Heyting algebra. It is another story to prove that the pseudometric
$\dist{L'}$ is precisely the native metric of $L'$, as we will do now.

\begin{theorem}\label{th:complete}
  Let $L$ be a Hausdorff co-Heyting algebra. Let $(L', \dist{}')$ 
  be the completion of the metric space $(L,\dist{L})$. Then $L'$ 
  is a Hausdorff co-Heyting algebra, and $\dist{}'$ is exactly the 
  ultrametric $\dist{L'}$. 
  Moreover for every positive integer $d$, $L'/dL'=L/dL$ and (using
  remark~\ref{re:identification-quotient}) $\pi_{dL}$ is the restriction
  of $\pi_{dL'}$ to $L$.
\end{theorem}

It is worthwhile to notice, before starting the proof, that the
``triangle inequality'' for $\bigtriangleup$ (see section~\ref{se:preprequesites})
implies that $\dist{L}$ is an {\em ultrametric}:
\begin{displaymath}
  \dist{L}(a,c) \leq \max \dist{L}(a,b),\dist{L}(b,c)
\end{displaymath}
It follows that a sequence $(x_n)_{n<\omega}$ is Cauchy if and only if 
$\dist{L}(x_n,x_{n+1})$ is convergent to $0$. 

\begin{proof}
Note that $\dist{L}(a,b)=\dist{L}(a \bigtriangleup b, \ZERO)$ for every $a,b \in
L$. By density it follows that:
\begin{equation}
  \label{eq:dbar-delta}
  \forall a',b' \in L',\quad \dist{}'(a',b')=\dist{}'(a' \bigtriangleup b', \ZERO)
\end{equation}
In order to show that $\dist{}'=\dist{L'}$ it is then sufficient to
check that they define the same balls centered at $\ZERO$. Since
$\dist{}'$ extends $\dist{L}$ and $L$ is dense in $L'$, the ball of
radius $2^{-d}$ and center $\ZERO$ for $\dist{}'$ is precisely the
closure $\overline{dL}$ of $dL$ in $L'$ with respect to $\dist{}'$. So
it suffices to check\footnote{Here we use that both $\dist{}'$ and
$\dist{L}$ take their values in $\{2^{-d}\}_{d<\omega}\cup\{0\}$. Indeed by the
ultrametric triangle inequality, for every $a' \neq b' \in L'$,
$\dist{}'(a',b')=\dist{}'(a,b)$ for some (any) $a,b \in L$ close enough
to $a',b'$.} that $dL'=\overline{dL}$ for every positive integer $d$.

The codimetric topology on $L/dL$ is discrete by
remark~\ref{re:dim-finie} so the metric of $L/dL$ is complete.
Moreover, by proposition~\ref{pr:morph-lipshitz}, $\pi_d$ is continuous.
Hence $\pi_d$ extends uniquely to a continuous map $\overline{\pi_d}:L' \to
L/dL$ which is an $\ltc$\--morphism by the same arguments as above
($\pi_d$ preserves $\ltc$\--equations hence so does $\overline{\pi_d}$ by
continuity). The kernel of $\overline{\pi_d}$ is the closure of $\Ker
\pi_d$, that is of $dL$, in $L'$ with respect to $\dist{}'$. This
morphism is surjective and $\dim L/dL < d$ so the points
(\ref{it:morph-lipshitz}) and (\ref{it:phi-de-dL}) of
corollary~\ref{co:TC-morph-codim} give us:
\begin{displaymath}
  dL \subseteq dL' \subseteq \overline{dL}
\end{displaymath}
Conversely let $a' \in \overline{dL}$. We show by induction on $d$ that
$a' \in dL'$.

If $d=0$ this is obvious since $0L'=L'$. So let us assume that $d \geq 1$
and $(d-1)L'$ is closed with respect to $\dist{}'$. Let $(a_n)_{n<\omega}$
be a sequence of elements of $dL$ converging to $a'$ with respect to
$\dist{}'$. Then $\dist{}'(a_n,a_{n+1})$ is convergent to 0 hence so
does $\dist{L}(a_n,a_{n+1})$, as $\dist{}'$ and $\dist{L}$
coincide on $L$. We may assume that $\codim a_n \bigtriangleup a_{n+1} \geq n+1$
for every $n$, by taking a subsequence of $(a_n)_{n<\omega}$ if necessary.
So by theorem~\ref{th:dim-et-rang} we can find $x_n \in nL$ such that
$a_n \bigtriangleup a_{n+1} \leq x_n$.

Since $a_d \in dL$ we can find $b_d \in (d-1)L$ such that $a_d \ll b_d$. 
For every $k \leq d$ let $b_k=b_d$. Assume that for some $n \geq d$ we have 
constructed a sequence $b_0,\dots,b_d$ of elements of $(d-1)L$ so that 
$a_n \ll b_n$ and $b_{n-1} \bigtriangleup b_n \in nL$.
Let $b_{n+1}=b_n \join x_{n+1}$. Since $x_{n+1} \in (n+1)L \subseteq (d-1)L$, by
construction $b_{n+1} \in (d-1)L$. Moreover:
\begin{displaymath}
  a_{n+1} \meet a_n \leq a_n \ll b_n \leq b_{n+1} 
\end{displaymath}
\begin{displaymath}
  a_{n+1} - a_n \leq a_n \bigtriangleup a_{n+1} \ll  x_{n+1} \leq b_{n+1}
\end{displaymath}
So $a_{n+1} = (a_{n+1} - a_n) \join (a_{n+1} \meet a_n) \ll b_{n+1}$.
Finally $b_n \bigtriangleup b_{n+1} \leq x_{n+1}$ hence $b_n \bigtriangleup b_{n+1} \in (n+1)L$.

So we can continue this construction by induction. It gives a sequence
$(b_n)_{n < \omega}$ of elements of ${(d-1)L}$ such that 
$a_n \ll b_n$ for every $n$. 
Moreover $\dist{L}(b_n,b_{n+1}) \leq 2^{-n-1}$ hence this is a Cauchy
sequence. Let $b'$ be its limit in $L'$ with respect to $\dist{}'$. 
By the induction hypothesis $\overline{(d-1)L}=(d-1)L'$ hence 
$b' \in (d-1)L'$, that is $\codim_{L'} b' \geq d-1$.
Since $a_n \join b_n = b_n$ and $b_n - a_n = b_n$ for every $n$,  the same
holds for $a',b'$ by continuity. So $a' \leq b'$ and $b'-a'=b'$
that is $a' \ll b'$. By theorem~\ref{th:dim-et-rang} we
conclude that $\codim_{L'} a' \geq d$ that is $a' \in dL'$. 

This ends the proof that $dL'=\overline{dL}$ for every positive
integer $d$. It follows that $\dist{L'}=\dist{}'$. In particular
$\dist{L'}$ is a metric on $L'$. Moreover $L'/dL'=L/dL$ since 
$\Ker \pi_{dL'} = \Ker \overline{\pi_{dL}}$.
\end{proof}

As in section~\ref{se:pre-compact}, for every co-Heyting algebra
$L$, let $\widehat{L}$ denote the limit of the projective system:
\begin{displaymath}
  \cdots \to L/(d+1)L \to L/dL \to \cdots \to L/0L = \{\ZERO\} 
\end{displaymath}
with projections $\pi_{d,d+1}:L/(d+1)L \to L/dL$. Recall that $\widehat
L$ can be represented as: 
\begin{displaymath}
  \widehat{L} = \bigl\{ (x_d)_{d<\omega} \tq \forall d,\ x_d \in L/dL \mbox{ and }
    \pi_{d,d+1}(x_{d+1})=x_d \bigr\}
\end{displaymath}

\begin{theorem}\label{co:completion-pro-lim}
  Let $L$ be a co-Heyting algebra. Then $\widehat L$ is the
  Hausdorff completion of $L$, and the projective topology on
  $\widehat{L}$ coincides with its codimetric topology.
\end{theorem}

\begin{proof}
Let $L_0=L/\omega L$ be the largest Hausdorff quotient of $L$. 
Then $L$ is also the completion of $L_0$, and $L_0/dL_0=L/dL$ for every
positive integer $d$. So we may assume that $L=L_0$, 
that is $L$ is Hausdorff.

Let $L'$ be the completion of $L$. We know that $L'/dL'=L/dL$ for
every positive integer $d$ by theorem~\ref{th:complete}. 
So $x \mapsto (\pi_{dL'}(x))_{d<\omega}$ defines an $\ltc$\--morphism 
$\varphi:L' \to \widehat{L}$ whose restriction to $L$ is the canonical
embedding of $L$ in $\widehat{L}$.

$\Ker \varphi = \bigcap_{d<\omega} \Ker \pi_{dL'} =\{\ZERO\}$ hence $\varphi$ is
injective. In order to show that it is surjective let us take any
element $y=(y_d)_{d<\omega}$ in the projective limit. Then each
$y_d=\pi_{dL}(x_d)$ for some $x_d \in L$. Since $\pi_{d,d+1}(y_{d+1})=y_d$
and $\pi_{d,d+1}\circ\pi_{(d+1)L}=\pi_{dL}$ we have $\pi_{dL}(x_{d+1})=\pi_{dL}(x_d)$
so $x_d \bigtriangleup x_{d+1}\in dL$. It follows that $(x_d)_{d<\omega}$ is a
Cauchy sequence in $L$ hence it converges to some $x \in L'$. 

\begin{displaymath}
  \pi_{dL'}(x)=\lim \pi_dL(x_n)=\pi_{dL'}(x_d)
\end{displaymath}
So $\varphi(x) = (\pi_{dL'}(x))_{d<\omega} = (\pi_{dL'}(x_d))_{d<\omega} = (y_d)_{d<\omega} =
y$. This ends the proof that $\varphi$ is an $\ltc$\--isomorphism. By 
theorem~\ref{th:complete} it follows that
$\widehat{L}/d\widehat{L}=L/dL$ (see
remark~\ref{re:identification-quotient}) for every positive integer
$d$. So the projective topology of $\widehat{L}$ coincides with the
codimetric topology.
\end{proof}

\begin{remark}
  A quotient of a co-Heyting algebra $L$ by an ideal $I$ is finite
  dimensional if and only if $dL \subseteq I$ for some $d$ (see
  corollary~\ref{co:TC-morph-codim}). So $\widehat L$ is also the
  limit of the projective system of {\em all} finite dimensional
  quotients of $L$. 
\end{remark}

\begin{corollary}\label{pr:partie-compact}
  A subset $X$ of a complete Hausdorff co-Heyting algebra $L$ is
  compact if and only if is closed and $\pi_d(X)$ is finite for every
  positive integer $d$.
\end{corollary}

\begin{proof}
If $X$ is compact it is obviously closed. Moreover for any positive
integer $d$ the sets $U(x,d)=\{y \in L \tq x \bigtriangleup y \in dL\}$ form an open cover
of $X$ as $x$ ranges over $X$. By compactness there is a finite subset
$X_d$ of $X$ such that $\{U(x,d)\}_{x \in X_d}$ covers $X$. Then
$\pi_d(X)=\pi_d(X_d)$ is finite.

Conversely since $L = \widehat{L}$ by
corollary~\ref{co:completion-pro-lim}, the topological closure of $X$ is
known to be:
\begin{displaymath}
  \overline{X} = \{ (x_k)_{k<\omega} \in \widehat{L} \tq \forall k,\ x_k \in \pi_k(X)\} 
\end{displaymath}
So if $X$ is closed and every $\pi_k(X)$ is finite then $X =
\overline{X}$ is compact as the limit of a projective system of
finite discrete spaces.
\end{proof}

A pseudometric space is called {\df precompact} if and only if
its Hausdorff completion is compact. The following corollary, which
immediately follows from corollaries~\ref{co:completion-pro-lim} and
\ref{pr:partie-compact} justifies
our terminology for precompact co-Heyting algebras.

\begin{corollary}\label{co:complet-compact}
  The Hausdorff completion of a co-Heyting algebra $L$ is compact if
  and only if $L/dL$ is finite for every positive integer $d$. 
\end{corollary}

We conclude with two delightful results which show that some metric
properties of complete co-Heyting algebra have a familiar flavour.
Recall that a sequence $(x_n)_{n < \omega}$ in a pseudometric space
$(X,\dist{})$ is convergent to $y$ if and only if $\dist{L}(x,y)$ is
convergent to 0. The uniqueness of the limit holds only in the
Hausdorff case. 

\begin{theorem}\label{th:gendarmes}
  Consider three sequences in a co-Heyting algebra $L$ such that $c_n
  \leq b_n \leq a_n$ for every $n < \omega$. If $a_n$ and $c_n$ converge to the
  same limit $l$ then $b_n$ is convergent to $l$.
\end{theorem}

\begin{proof}
Let $u_n = (a_n - l) \join (c_n -l)$, this sequence is convergent to
$\ZERO$ (by continuity of the terms). By assumption $b_n \bigtriangleup l \leq u_n$
hence $\codim b_n \bigtriangleup l \geq \codim u_n$ for every positive integer $n$. So
$\dist{L}(b_n,l) \leq \dist{L}(u_n,\ZERO)$ is convergent to 0. 
\end{proof}

\begin{corollary}\label{co:CV-monotone}
  Every monotonic sequence in a compact subset $X$ of a 
 co-Heyting algebra $L$ is convergent.
\end{corollary}

\begin{proof}
Let $(a_n)_{n<\omega}$ be a monotonic sequence in $X$.
Let $(a_{\sigma(n)})_{n<\omega}$ a subsequence convergent in $X$. If
$(a_n)_{n<\omega}$ is increasing, for every integer $k$
let $n_k$ be the smallest integer $n$ such that $a_k \leq a_{\sigma(n)}$. 
\begin{equation}\label{eq:comp-croissante}
  a_{\sigma(n_k-1)} \leq a_k \leq a_{k+1} \leq a_{\sigma(n_{k+1})}
\end{equation}
Conversely if $(a_n)_{n<\omega}$ is decreasing let $n_k$ be the smallest integer $n$
such that $a_k \geq a_{\sigma(n)}$. We have the same inequalities as in
(\ref{eq:comp-croissante}) with reverse order. In both cases 
$a_{\sigma(n_k-1)}$ and $a_{\sigma(n_{k+1})}$ converge to the same limit hence
so does $a_k$ by theorem~\ref{th:gendarmes}.
\end{proof}

\section{Appendix}
\label{se:appendix}

Proposition~\ref{pr:fg-free-TC-codimetric} allows a slight improvement
of the finite model property (to be compared with fact~\ref{fa:fmp-HA}).

\begin{proposition}\label{pr:FMP-formules}
  Let $\cV$ be a variety of co-Heyting algebras having the finite
  model property and $\theta(x)$ be a quantifier free $\ltc$\--formula. If
  there exists a $\cV$\--algebra $L$ such that $L \models \exists x\;\theta(x)$ then
  there exists a finite $\cV$\--algebra having this property.
\end{proposition}

\begin{proof}
We may assume that $\theta(x)$ is a conjunction of atomic and negatomic
formulas with $n$ variables. Since $t(x) \leq t'(x)$ is equivalent,
modulo the theory of co-Heyting algebras, to $t(x)-t'(x)=\ZERO$, we
can suppose that every atomic formula is of type $t(x)=\ZERO$.
Finally $t(x)=\ZERO$ and $t'(x)=\ZERO$ is equivalent to $t(x) \join
t'(x) = \ZERO$ so we can assume:
\begin{displaymath}
  \theta(x) \equiv \cconj_{i \leq r} t_i(x)\neq\ZERO \conj t(x)=\ZERO 
\end{displaymath}
Let $a$ be a tuple of elements of $L$ such that $L\models\theta(a)$. We may
assume that $L$ is generated by $a$. Let $\cF_n$ be the free
$\cV$\--algebra with $n$ generators and $\pi:\cF_n \to L$ the projection
which maps the free generators $X$ of $\cF_n$ onto $a$. Let
$(g_k)_{k<\omega}$ be an enumeration of the kernel of $\pi$.  By construction
$t(X) = g_k(X)$ for some $k$, but $t_i(X) \not\leq g_l(X)$ for every
positive integer $l$ and every $i\leq r$. By
proposition~\ref{pr:fg-free-TC-codimetric}, $\cF_n$ is Hausdorff so:
\begin{displaymath}
  \max_{i\leq r} \codim t_i(X)-g_k(X) < \omega 
\end{displaymath}
Let $d$ denote this integer. Let $I$ be the ideal of $\cF_n$ generated
by $g_k(X)$ and $\varepsilon_{d+1}(\cF_n)$, and let $b$ be the image of $X$ in
$\cF_n/I$ {\it via} the canonical projection.
By construction $\cF_n/I$ is a quotient of $\cF_n/(d+1)\cF_n$. By 
proposition~\ref{pr:fg-free-TC-codimetric} and the assumption on $\cV$, 
$\cF_n/(d+1)\cF_n$ is finite hence so is $\cF_n/I$. Moreover $t(X)$ 
belongs to $I$ and none of the $t_i(X)$'s belongs to $I$ so 
$\cF_n/I \models \theta(b)$.
\end{proof}

We have seen that if a co-Heyting algebra $L$ is finitely presented, then 
$dL$ is a principal ideal for every positive integer $d$
(corollary~\ref{co:fg-precomp} and lemma~\ref{le:precomp-isole}). This is
actually true for finitely generated co-Heyting algebras, and even
more is true:

\begin{proposition}
  For every positive integers $n,d$ there exists an $\ltc$\--term
  $t_{n,d}$ in $n$ variables such that for every co-Heyting
  algebra $L$ generated by some $a \in L^n$, $t_{n,d}(a)=\varepsilon_d(L)$.
\end{proposition}

\begin{proof}
Let $t_{n,d}$ be an $\ltc$\--term such that in the free
co-Heyting algebra $\cF_n$ generated by an $n$\--tuple $X$, 
$t_{n,d}(X)=\varepsilon_d(\cF_n)$. Let $L$ be any co-Heyting algebra
generated by some $n$\--tuple $a$ and $\varphi$ the projection of $\cF_n$
onto $L$ which maps $X$ onto $a$.
By corollary~\ref{co:TC-morph-codim}(\ref{it:phi-de-epsilon}) 
$\varphi(\varepsilon_d(\cF_n)) = \varepsilon_d(L)$ so: 
\begin{displaymath}
   t_{n,d}(a) = \varphi(t_{n,d}(X)) = \varphi(\varepsilon_d(\cF_n)) = \varepsilon_d(L)
\end{displaymath}
\end{proof}

\begin{remark}\label{re:formule-tnd}
  Our approach does not give any explicit form for $t_{n,d}$. Such an
  expression can be derived from Bellissima's construction. Indeed
  an explicit formula for all the join irreducible elements of fixed
  dimension $d$ in the free co-Heyting algebra $\cF_n$ with $n$
  generators is provided by this construction (see \cite{bell-1986},
  or theorem~3.3 in \cite{darn-junk-2008} for a slightly better
  formula). Their join gives an expression for $t_{n,d}$, but its
  complexity seems to be discouraging for practical computations. 
\end{remark}

Let $\cV_{n,d}$ be the variety of co-Heyting algebras axiomatized by
the equation $t_{n,d+1}=\ZERO$. This is the variety of co-Heyting
algebras $L$ such that every subalgebra of $L$ generated by $n$
elements has dimension at most $d$. So a variety $\cV$ is contained in
$\cV_{n,d}$ if and only if the algebra freely generated in $\cV$ by $n$
elements has dimension at most $d$. Of course a variety $\cV$ of
co-Heyting algebras is locally finite (that is every finitely generated
algebra in $\cV$ is finite) if and only if for every positive integer
$n$ there is an integer $d(n)$ such that $\cV \subseteq \cV_{n,d(n)}$. For
every $n \geq 1$, $\cV_{n,0}$ is nothing but the variety of boolean
algebras, hence it is locally finite. On the other hand one can easily
show by adapting an example of Mardaev~\cite{mard-1984} that the
varieties $\cV_{1,d}$ for $d >1$ are distinct and not locally finite.

\begin{question}\label{qu:loca-fin}
  For which integers $n,d$ is $\cV_{n,d}$ locally finite? 
\end{question}

It is asked in \cite{bezh-grig-2005} if $\cV$ is a locally finite
variety whenever the algebra freely generated in $\cV$ by 2 elements
is finite. This is equivalent to the local finiteness of $\cV_{2,d}$
for every $d$, and it would imply that $\cV_{n,d}$ is locally finite
for every $n \geq 2$ and every $d$ because $\cV_{n,d}$ is obviously
contained in $\cV_{2,d}$.

\phantomsection
\addcontentsline{toc}{section}%
   {References}
% \bibliographystyle{alpha}
% \bibliography{biblio}

\newcommand{\etalchar}[1]{$^{#1}$}

\end{document}